\documentclass[10pt]{amsart}

\usepackage{amsfonts}
\usepackage{amssymb}
\usepackage{color}
\usepackage{amsmath, amsthm, latexsym}
\usepackage[all]{xy}
\usepackage{hyperref}
\usepackage{graphicx}
\usepackage[margin=1.5in]{geometry}

\def \c{\mathbb{C}}
\def \z{\mathbb{Z}}
\def \r{\mathbb{R}}
\def \n{\mathbb{N}}
\def \p{\mathbb{P}}
\def \q{\mathbb{Q}}
\def \A{\mathbb{A}}

\def \V{\mathcal{V}}
\def \K{\mathcal{K}}
\def \O{\mathcal{O}}

\def \k{{\bf k}}

\def \L{{\bf L}}
\def \G{\mathcal{G}}

\def \trop{\textup{trop}}
\def \Trop{\textup{Trop}}
\def \TROP{{\bf Trop}}
\def \Log{\textup{Log}}
\def \In{\textup{in}}
\def \ord{\textup{ord}}
\def \Proj{\textup{Proj}}
\def \gr{\textup{gr}}

\def \Spec{\textup{Spec}}
\def \Hom{\textup{Hom}}
\def \supp{\textup{supp}}
\def \SL{\textup{SL}}
\def \GL{\textup{GL}}
\def \conv{\textup{conv}}
\def \an{\textup{an}}
\def \hc{\textup{hc}}
\def \aff{\textup{aff}}
\def \Lie{\textup{Lie}}
\def \hom{\textup{hom}}

\theoremstyle{plain}
\newtheorem{Th}{Theorem}[section]
\newtheorem{Lem}[Th]{Lemma}
\newtheorem{Prop}[Th]{Proposition}
\newtheorem{Cor}[Th]{Corollary}

\newtheorem{THM}{Theorem}
\newtheorem{CORO}[THM]{Corollary}

\theoremstyle{definition}
\newtheorem{Ex}[Th]{Example}
\newtheorem{Def}[Th]{Definition}
\newtheorem{Rem}[Th]{Remark}
\newtheorem{Ass}[Th]{Assumption}

\newtheorem*{REM}{Remark}

\begin{document}
\title{Gr\"obner theory and tropical geometry on spherical varieties}

\author{Kiumars Kaveh}

\author{Christopher Manon}

\maketitle

\begin{abstract}
Let $G$ be a connected reductive algebraic group.
We develop a Gr\"obner theory for multiplicity-free $G$-algebras, as well as a 
tropical geometry for subschemes in a spherical homogeneous space $G/H$. We define the notion of a 
spherical tropical variety and prove a fundamental theorem of tropical geometry in this context. We also propose a definition for a spherical amoeba in $G/H$. 
Our work partly builds on the previous work of Vogiannou on spherical tropicalization and in some ways is complementary.  
\end{abstract}

\setcounter{tocdepth}{1}
\tableofcontents

\section*{Introduction}
Let $\k$ be an algebraically closed field of characteristic $0$.\footnote{The assumption that $\k$ is algebraically closed and characteristic $0$ is not needed in several of the results in the paper.}
This paper extends the Gr\"obner theory of ideals in a polynomial algebra $\k[x_1, \ldots, x_n]$, as well as the tropical geometry of 
subvarieties in the algebraic torus $(\k^*)^n$, to the setting of spherical varieties for an action of a connected reductive algebraic group $G$ over $\k$.\footnote{Throughout $\k^* = \k \setminus \{0\}$ denotes the multiplicative group of $\k$.}

As far as the authors know, the problem of developing tropical geometry on spherical varieties was first suggested by Gary Kennedy (\cite{Kennedy-talk}).
The first results in this direction appeared recently in the interesting paper of Vogiannou (\cite{Vogiannou}).
In his paper, Vogiannou defines a tropicalization map for a spherical homogeneous space and extends the results in \cite{Tevelev} to the spherical setting. We should also mention \cite{Nash} which suggests a notion of tropicalization of a spherical embedding. We are not aware of any previous work on Gr\"obner theory for spherical varieties or, in other words, for multiplicity-free $G$-algebras. 

First, let us briefly review spherical varieties as well as classical Gr\"obner theory and tropical geometry. 

Spherical varieties are a generalization of toric varieties for actions of reductive groups. Let $G$ be a connected reductive algebraic group. 
A variety $X$ with an action of $G$ (i.e. a $G$-variety) is called {\it spherical} if a Borel subgroup $B$ (and hence any Borel subgroup) of $G$ has a dense orbit.\footnote{As part of the definition, a spherical variety is usually assumed to be normal, but we will not need the normality assumption in the paper unless otherwise stated.} 
If $X$ is spherical it has a finite number of $G$-orbits as well as a finite number of 
$B$-orbits. Generalizing toric varieties, the geometry of spherical varieties 
can be read off from associated convex polytopes and convex cones. In particular, if $G/H$ is a spherical homogeneous space, the celebrated 
Luna-Vust theory gives a one-to-one correspondence between spherical $G$-varieties containing $G/H$ as the dense $G$-orbit (i.e. spherical embeddings of $G/H$) 
and the so-called colored fans (\cite{Luna-Vust, Knop-LV}). It is a well-known fact that if $L$ is a $G$-linearized line bundle on a spherical variety $X$ then the space of sections $H^0(X, L)$ is a multiplicity free $G$-module.
In particular, the ring of regular functions $\k[X]$ of a spherical variety is a multiplicity-free $G$-algebra. 
 Many important classes of varieties are in fact spherical. Toric varieties are exactly spherical varieties for when $G=T$ is an algebraic torus.
The flag variety $G/B$ and the partial flag varieties $G/P$ are spherical by the Bruhat decomposition. Another example is the variety of smooth 
quadrics in a projective space. This example plays an important role in enumerative geometry. In Section \ref{sec-spherical} we discuss some background material 
about spherical varieties. For a nice overview of the theory of spherical varieties we refer the reader to \cite{Perrin}.

In the usual Gr\"obner theory for ideals in a polynomial ring $\k[x_1, \ldots, x_n]$ (see Section \ref{subsec-Grobner-prelim}), 
one begins by fixing a total order $\succ$ on 
the additive semigroup $\z^n_{\geq 0}$, e.g. a (reverse) lexicographic order. 
Given an ideal $I \subset \k[x_1, \ldots, x_n]$ one then defines the initial ideal $\In_\succ(I)$ with respect to $\succ$.
Many of the properties of $I$ and its initial ideal $\In_\succ(I)$ are 
related. The ideal $\In_\succ(I)$ is a monomial ideal and hence has the advantage that its structure can be described combinatorially. 
One of the first main results in Gr\"obner theory is that a given ideal has only a finite number of initial ideals, for all possible choices of total orders $\succ$. A Gr\"obner basis for $I$ is a finite collection of elements in $I$ whose initial monomials generate $\In_\succ(I)$. Given a Gr\"obner basis for an ideal $I$ one can 
give effective algorithms to solve many computational problems concerning $I$. For example, one can solve the ideal membership problem, that is, to decide whether a given polynomial lies in $I$ 
or not. The celebrated Buchberger algorithm produces a Gr\"obner basis for $I$ starting from a set of ideal generators.

Similarly, given a vector $w \in \r^n$ and an ideal $I \subset \k[x_1, \ldots, x_n]$, 
one defines the initial ideal $\In_w(I)$.
For $w_1, w_2 \in \r^n$ one says that $w_1 \sim w_2$ if the initial ideals $\In_{w_1}(I)$ and $\In_{w_2}(I)$ coincide. Another important result in Gr\"obner theory asserts that, when $I$ is a {homogeneous ideal}, the closures of equivalence classes of the relation $\sim$ are convex rational polyhedral cones. The resulting fan is called 
the {\it Gr\"obner fan} of $I$ and is an important concept in the theory (\cite{Mora-Robbiano} and \cite[Chapters 1 and 2]{Sturmfels}).

The tropical variety of a subvariety of $(\k^*)^n$ (i.e. a very affine variety) is a polyhedral fan in $\r^n$ that encodes the asymptotic directions in the subvariety (see Section \ref{subsec-tropical-prelim}). 
There are basically two ways to define the tropical variety of a subvariety $Y \subset (\k^*)^n$. The first way is to use initial ideals. In fact, more generally one defines the 
tropical variety of an ideal. Let $I \subset \k[x_1^\pm, \ldots, x_n^\pm]$ be an ideal. 
The tropical variety $\trop(I)$ is defined as the set of all $w \in \r^n$ such that $\In_w(I)$ does not 
contain any monomials. 
The other way to define the tropical variety of a subvariety is to use Puiseux series and the tropicalization map. Let $\K = \k((t))$ denote the field of formal Laurent series in 
a variable $t$ and let $\overline{\K}$ denote its algebraic closure, that is, the field of formal Puiseux series. One defines the tropicalization map 
$\Trop: (\overline{\K}^*)^n \to \q^n$ by:
$$\Trop(\gamma_1, \ldots, \gamma_n) = (\ord_t(\gamma_1), \ldots, \ord_t(\gamma_n)),$$
where $\ord_t$ denotes the order of $t$, i.e. the exponent of the lowest term in $t$ of a Puiseux series. It is a natural valuation that comes with the field of Puiseux series $\overline{\K}$.
Let $Y \subset (\k^*)^n$ be a subvariety and let $Y(\overline{\K})$ denote its set of points over $\overline{\K}$.
The tropical variety of the subvariety $Y$ is then defined to be the closure of $\Trop(Y(\overline{\K}))$ in $\r^n$. 
The fundamental theorem of tropical algebraic geometry asserts that the above two constructions coincide. That is, if $I = I(Y)$ is the ideal of the subvariety
$Y \subset (\k^*)^n$ then $\trop(I) = \Trop(Y)$. 

One shows that the tropical variety of an ideal is the support of a rational polyhedral fan in $\r^n$. More generally, one can define the notion of tropical variety for 
subvarieties in an algebraic torus that are defined over a field extension $K$ of $\k$ where $K$ is equipped with a valuation (e.g. the field of Puiseux series 
$\overline{\K}$ together with the valuation $\ord_t$ as above). In this more general setting, a tropical variety is a polyhedral complex instead of just a polyhedral fan.
 
After these brief reviews, let us give a summary of the contributions of the present paper. Let $G$ be a connected reductive algebraic group over $\k$.

In Section \ref{sec-Grobner-G-alg}, we extend several basic definitions and 
results from Gr\"obner theory in a polynomial algebra to the spherical setting. We consider a quasi-affine spherical $G$-variety $X$ with
 $A=\k[X]$ its algebra of regular functions.
 
Extending the toric case, to a spherical variety $X$ one associates a sublattice $\Lambda_X$ of the weight lattice $\Lambda$ of $G$. It is the lattice of all the 
weights of $B$-eigenfunctions in the field of rational functions $\k(X)$. Analogous to the well-known notion of dominant weight order on $\Lambda$, we define 
a partial order $>_X$ in $\Lambda_X$ which we call the spherical dominant order (see paragraph before Theorem \ref{th-coor-ring-multiplication-spherical}). 
We consider total orders which refine the spherical dominant order $>_X$. For such a total order $\succ$ we define the associated graded algebra $\gr_\succ(A)$ 
which is in fact isomorphic to $A_\hc$, the horospherical contraction of $A$ (see \cite{Popov} for the notion of horospherical contraction).\footnote{We recall that a $G$-variety is horospherical if $G$-stabilizer of any point contains a maximal unipotent subgroup.} 

Also for an ideal $I \subset A$ we define its initial ideal $\In_\succ(I)$ which is an ideal in $\gr_\succ(A)$. We say that a subset $\G \subset I$ is a
{\it spherical Gr\"obner basis} if the image of $\G$ in $\In_\succ(I)$ generates this ideal (Definition \ref{def-sph-Grobner-basis}). We give a generalization of the well-known division algorithm to this setting (Proposition \ref{prop-sph-div-algo}) and prove that a spherical Gr\"obner basis $\G$ is a set of ideal generators for the original ideal $I$ (Proposition \ref{prop-sph-Grobner-basis-generate-ideal}). Our first main result in Section \ref{sec-Grobner-G-alg} is the following.
\begin{THM} \label{th-intro-finite-number-initial-ideal}
An ideal $I \subset A$ has a finite number of initial ideals (regarded as ideals in the horospherical contraction $A_{\hc}$ of $A$).
\end{THM}
This theorem then implies the existence of a {\it universal spherical Gr\"obner basis} for $I$ (Corollary \ref{cor-universal-Grobner-basis}).

Next we consider generalizations of the notions of initial ideal with respect to a vector $w \in \r^n$ and Gr\"obner fan of an ideal.
In the spherical Gr\"obner theory the role of a vector $w \in \r^n$ is played by a $G$-invariant valuation on the field of rational functions $\k(X)$ (see Section \ref{subsec-inv-valuation}).
We let $\V_X$ denote the collection of all $G$-invariant valuations on $\k(X)$ and with values in $\q$. It is a well-known result that 
$\V_X$ can be naturally realized as a simplicial cone sitting in the dual vector space $\Hom(\Lambda_X, \q)$. Hence sometimes $\V_X$ is referred to as the {\it valuation cone}. Considering $G$-invariant valuations in the study of spherical varieties goes back to the fundamental paper of Luna and Vust (\cite{Luna-Vust}).

For a valuation $v \in \V_X$ we define the associated graded algebra $\gr_v(A)$. One makes the following observation (Proposition \ref{prop-partial-horo-degeneration-gr_v}): 
For $v \in \V_X$, the associated graded algebra $\gr_v(A)$ depends 
only on the face $\sigma$ of the cone $\V_X$ which contains $v$ in its relative interior (thus we also write $\gr_\sigma(A)$ instead of $\gr_v(A)$).
When $v$ lies in the interior of $\V_X$ then $\gr_v(A)$ is isomorphic to the horospherical contraction $A_{\hc}$.\footnote{This statement is not quite new and has already been observed by other authors.}

For an ideal $I \subset A$ and a valuation $v \in \V_X$ we consider the initial ideal $\In_v(I) \subset \gr_v(A)$. Extending the usual Gr\"obner theory we define 
an equivalence relation on the valuations: for $v_1, v_2 \in \V_X$ we say $v_1 \sim v_2$ if they lie on the relative interior of the same face 
$\sigma$ of $\V_X$ and also $\In_{v_1}(I) = \In_{v_2}(I)$ regarded as ideals in $\gr_\sigma(A)$ (Definition \ref{def-equiv-valuation}). 

For the next result we assume that $A = \bigoplus_{i \geq 0} A_i$ is graded and $G$ acts on $A$ preserving the grading. Moreover, we assume that each graded piece $A_i$ is a multiplicity free $G$-module (i.e. $A$ is a multiplicity free $(\k^* \times G)$-algebra). Thus $A$ is the ring of regular functions on an affine spherical $(\k^* \times G)$-variety $X$. In this situation by the valuation cone $\V_X$ we mean the cone of $(\k^* \times G)$-invariant valuations. We prove the following (Theorem \ref{th-sph-Grobner-fan}).
\begin{THM}  \label{th-intro-Grobner-fan}
Let $A$ be graded as above. Let $I \subset A$ be a homogeneous ideal. Then the closures of equivalence classes of $\sim$ form a fan 
which we call the spherical Gr\"obner fan of $I$.
\end{THM}


We would like to point out some important differences between the toric case and the general spherical case which makes the spherical theory more 
complicated. (1)
In the torus case the isotypic components are $1$-dimensional corresponding to different Laurent monomials, while in the general spherical case they 
are irreducible $G$-modules and almost always have dimension greater than $1$. (2) If $f_\alpha = x^\alpha$, $f_\beta=x^\beta$ are two monomials in a polynomial algebra then $f_\alpha f_\beta = f_{\alpha +  \beta}$. In the spherical case, if $f_\gamma \in W_\gamma$, $f_\mu \in W_\mu$ where $A = \bigoplus_\lambda W_\lambda$ is the isotypic decomposition of the $G$-algebra $A$, then in general, $f_\gamma f_\mu$ does not necessarily lie in $W_{\gamma+\mu}$ but rather in $W_{\gamma+\mu}$ direct sum with $W_\lambda$ where 
$\lambda$ is greater than $\gamma+\mu$ in the spherical dominant order $>_X$ (Theorem \ref{th-coor-ring-multiplication-spherical}).

In Sections \ref{sec-sph-trop-var-ideals} and \ref{sec-sph-trop-var-trop-map} we extend several basic definitions and results from 
tropical geometry to the spherical setting. We consider a spherical homogeneous space $G/H$ (which may not be quasi-affine).
The spherical tropical veriety encodes the ``asymptotic directions'' of a subvariety of a spherical homogeneous space $G/H$ which correspond to 
$G$-equivariant embeddings of $G/H$. To make the notion of asymptotic directions in $G/H$ precise, one uses $G$-invariant valuations. 
Extending the tropical geometry on the algebraic torus $(\k^*)^n$, we consider two ways of constructing a spherical tropical variety. One using the spherical tropicalization map and the other using  initial ideals. 

Given a valuation $v: \k(G/H) \to \q \cup \{\infty\}$ one can define a $G$-invariant valuation $\bar{v} \in \V_{G/H}$. For $0 \neq f \in \k(G/H)$, 
the value $\bar{v}(f)$ is equal to $v(g \cdot f)$ for $g$ in a Zariski open subset $U_f \subset G$. Also a formal curve
$\gamma$ in $G/H$, that is, a $\overline{\K}$-valued point $\gamma \in G/H(\overline{\K})$, defines a valuation $v_\gamma$. For $f \in \k(G/H)$, the 
value $v_\gamma(f)$ is simply $\ord_t$ of the restriction of $f$ to the curve $\gamma$. Combining these two constructions, to a formal curve 
$\gamma$ on $G/H$ one associates an invariant valuation $\bar{v}_\gamma \in \V_{G/H}$. This plays an important role in the Luna-Vust theory of spherical embeddings
(\cite{Luna-Vust, Knop-LV}). Vogiannou, suggests the map
$$\Trop: G/H(\overline{\K}) \to \V_{G/H}, \quad \gamma \mapsto \bar{v}_\gamma,$$ 
as a generalization of the tropicalization map to the setting of spherical varieties. If $Y \subset G/H$ is a subvariety, 
the {\it spherical tropicalization} $\Trop(Y)$ is then defined to be the image of $Y$ in $\V_{G/H}$ under the map $\Trop$ (see \cite{Vogiannou} and Definition \ref{def-sph-tropicalization}). He shows the following: For a valuation $v \in \V_{G/H}$ let $X_v$ denote the $G$-equivariant spherical embedding corresponding to the single ray generated by $v$ (in the sense of Luna-Vust). Then $v$ lies in $\Trop(Y)$ if and only if the closure of $Y$ in $X_v$ intersects the unique $G$-invariant divisor at infinity $D_v \subset X_v$.
It is also proved in \cite[Section 4]{Vogiannou} that $\Trop(Y)$ is the support of a rational polyhedral fan in $\V_{G/H}$. Vogiannou uses $\Trop(Y)$ to prove the existence of a so-called spherical tropical compactification for $Y$, extending an analogous result in \cite{Tevelev} for algebraic torus. The arguments rely on the Luna-Vust theory of spherical embeddings. In Section \ref{subsec-tropicalization} we review the spherical tropicalization map of Vogainnou. 

One of the goals of the present paper is to give a definition of a spherical tropical variety using the defining ideal of the subvariety. Note that a spherical homogeneous space $G/H$ may not in general be affine or quasi-affine. It turns out that in the spherical context, it is more natural to consider subvarieties in the open Borel orbit. The open Borel orbit is always an affine variety. 

Fix a Borel subgroup $B$ in $G$ and let $J \subset \k[X_B]$ be an ideal in the coordinate ring of the open $B$-orbit $X_B \subset {G/H}$. For $v \in \V_{G/H}$, the notions of associated graded $\gr_v(\k[X_B])$ and initial ideal $\In_v(J)$ are defined in a similar fashion as in the case of a $G$-algebra. 
We define $\trop_B(J)$ to be the set of all $v \in \V_{G/H}$ such that $\In_v(J)$ contains no unit elements (Definition \ref{def-sph-trop-var-X_B}). That is, 
$$\trop_B(J) = \{ v \in \V_{G/H} \mid \In_v(J) \neq \gr_v(\k[X_B])\},$$
(cf. Definition \ref{def-trop-var-ideal}). 

Take a $G$-linearized very ample line bundle $L$ on a projective spherical embedding $\overline{X}$ of $G/H$. Let $A = \bigoplus_{i \geq 0}H^0(\overline{X}, L^{\otimes i})$ be the algebra of sections of $L$. We apply the spherical Gr\"obner theory discussed earlier to the $(\k^* \times G)$-algebra $A$ to prove the following (Theorem \ref{th-fan-st-sph-trop-var}): 

\begin{THM}  \label{th-intro-trop-var-fan-st}
$\trop_B(J)$ is the support of a rational polyhedral fan. Moreover, this fan structure can be obtained by intersecting $\trop_B(J)$ with the spherical Gr\"obner fan of an ideal $\tilde{J} \subset A$ (which can be thought of as the homogenization of $J$ with respect to $(\overline{X}, L)$). 
\end{THM}

We also discuss how to compute the spherical tropical variety of a hypersurface (Section \ref{subsec-sph-trop-hypersurface}). Finally, we show the existence of a so-called {\it finite spherical tropical basis} for $J$. In other words, $\trop_B(J)$ is the intersection of a finite number of spherical tropical hypersurfaces (Section \ref{subsec-spherical-trop-basis}).

\begin{REM}
The above results in principle give an algorithmic way to compute the spherical tropical variety of a subvariety in a spherical homogeneous space. For example, in the case of a subvariety $Y$ of $\GL(n, \k)$, these suggest a way to compute the set of all the invariant factors of points of $Y$ over $\mathcal{K} = \k((t))$ (see Example \ref{ex-sph-var}(5) and Example \ref{ex-non-Arch-Cartan-decomp}). It might be worthwhile to work out the details and efficiency of such an algorithm.
\end{REM}

Let $Y \subset G/H$ be a subvariety and let $I=I(Y) \subset A$ the ideal of sections vanishing on $Y \subset G/H$.
In Section \ref{subsec-sph-trop-var-G/H} we put together the spherical tropical varieties in open Borel orbits, for different Borel subgroups, to define the notion of a tropical variety $\trop(I)$ (Definition \ref{def-sph-trop-var-G/H} and Proposition \ref{prop-sph-trop-var-G/H-vs-X_B}). 
For each Borel subgroup $B$ let $J_B \subset \k[X_B]$ denote the ideal of $Y \cap X_B$. The following is our generalization of the fundamental theorem of tropical geometry to the spherical setting (Theorem \ref{th-fundamental}). For $v \in \V_{G/H}$ let $X_v$ denote the spherical embedding of $G/H$ corresponding to the single ray generated by $v$.
\begin{THM} 
The following subsets of the valuation cone $\V_{G/H}$ coincide:
\begin{itemize}
\item[(a)] The set $\bigcup_B \trop(J_B),$ where the union is over all Borel subgroups of $G$ (one shows that
it is enough to take the union over a finite collection of Borel subgroups).
\item[(b)] The set $\Trop(Y) = \{\Trop(\gamma) \in \V_{G/H} \mid \gamma \in Y(\overline{\K}) \textup{ formal Puiseux curve on Y }\}.$ 
\item[(c)] The set $v \in \V_{G/H}$ such that the closure of $Y$ in $X_v$ intersects the divisor at infinity $D_v \subset X_v$.
\end{itemize}
\end{THM}
We remark that the equivalence of (b) and (c) is proved in \cite{Vogiannou}. We show the equivalence of (a) and (c). 


In Section \ref{subsec-analytification}, we propose a generalization of the usual tropicalization map, from the so-called analytification of a subvariety in $(\k^*)^n$ to $\r^n$, 
to the spherical setting. In this context it is natural to extend the valuation cone $\V_X$ and define $\hat{\V}_X$ to be the set of all invariant valuations $v: \k(X) \to \r \cup \{\infty\}$. Let $Y \subset X$ be a subvariety of an affine spherical $G$-variety $X$ that intersects the open $G$-orbit. Extending the spherical tropicalization map in \cite{Vogiannou}, we define $\TROP: Y^\an \to \hat{\V}_X$, where $Y^\an$ is the analytification of $Y$ in the sense of Berkovich. We like to regard the $\TROP$ map as an algebraic analogue of the averaging (over a compact group) map. We show that $\TROP$ is continuous (Proposition \ref{prop-trop-analytification}). This is a relatively straightforward extension. We believe that $\TROP: Y^\an \to \hat{\V}_X$ might give new insight into the notion of tropicalization.

\begin{REM}
We expect that the arguments in this paper extend to any valued field without much difficulty, i.e. when $Y \subset G/H$ is defined over a valued field extension $K$ of $\k$.
\end{REM}

Finally, we address the notion of amoeba of a subvariety.
When $\k = \c$, one defines a logarithm map on the torus $(\c^*)^n$ as follows. Fix a real number $t > 0$. The {\it logarithm map} 
$\Log_t: (\c^*)^n \to \r^n$ is simply defined by: $$\Log_t(z_1,  \ldots, z_n) = (\log_t(|z_1|), \ldots, \log_t(|z_n|)).$$
The {\it amoeba} of a subvariety $Y \subset (\c^*)^n$ is defined to be the image of $Y$ under the logarithm map. A well-known theorem 
states that as $t$ approaches $0$, the amoeba of $Y$ approaches, in Hausdorff metric, to the tropical variety of $Y$.
In Section \ref{sec-sph-amoeba} we propose a generalization of the notion of logarithm map (and amoeba) for spherical varieties. The {\it spherical logarithm map} is a $K$-invariant map ${\bf L}_t: G/H(\c) \to \V_{G/H}$, where $K$ is a maximal compact subgroup in $G$. 
For this we need to assume that the Archimedean Cartan decomposition holds for the spherical homogeneous space ${G/H}$ (Assumption \ref{ass-Arch-Cartan-sph}).
In fact, the authors originally conjectured that the Archimdean Cartan decomposition should hold for any spherical homogeneous space. Later, we learned that
Victor Batyrev had made the same conjecture some years ago (some related results can be found in \cite{Knop-Cartan}). 

In \cite{Eliyashev} a generalized notion of logarithm map (and amoeba) is introduced for a complex manifold $X$ with respect to a vector of holomorphic differentials. We expect that spherical logarithm map in the context of spherical varieties is a special case of this generalized logarithm map.

In Section \ref{sec-sph-amoeba} we give a generalization of the fact that amoeba approaches the tropical variety to the spherical setting (Theorem \ref{th-trop-var-limit-amoeba}):\footnote{We think that the stronger statement, that the spherical amoeba approaches the spherical tropical variety in Hausdorff distance, is true as well.}
\begin{THM}
Let $Y \subset {G/H}$ be a subvariety. Let $v \in \Trop(Y)$ be a point in the spherical tropical variety of $Y$. 
Then there exists a formal curve $\gamma \in Y(\overline{\K})$ with nonzero radius of convergence and such that: $$\lim_{t \to 0} \L_{t}(\gamma(t)) = v.$$ 
\end{THM}


As an interesting corollary we obtain a result which states that the Smith normal form of a matrix 
(whose entries are algebraic functions and regarded as Laurent series in one variable $t$) is a limit of the singular values of the matrix as 
$t$ approaches $0$ (Corollary \ref{cor-inv-fact-sing-values}). 
Recall that if $A$ is an $n \times n$ complex matrix, the singular value decomposition states that $A$ can be written as: $$A = U_1DU_2,$$ where $U_1$, $U_2$
are $n \times n$ unitary matrices and $D$ is diagonal with nonnegative real entries. In fact, the diagonal entries of $d$ are the eigenvalues of 
the positive semi-definite matrix $\sqrt{AA^*}$ where $A^* = \bar{A}^t$. The diagonal entries of $A$ are usually referred to as the {\it singular values of $A$}.
Now let $A(t)$ be an $n \times n$ matrix whose entires $A_{ij}(t)$ are Laurent series in $t$ (over $\c$). 
We recall that the Smith normal form theorem (over the ring of formal power series which is a PID) states that $A(t)$ can be written as:
$$A_1(t) \tau(t) A_2(t),$$
where $A_1(t)$, $A_2(t)$ are $n \times n$ matrices with power series entries and invertible over the power series ring, and $\tau(t)$ is a diagonal matrix of the form 
$\tau(t) = \textup{diag}(t^{v_1}, \ldots, t^{v_n})$ for integers $v_1, \ldots, v_n$. The integers $v_1, \ldots, v_n$ are usually called the 
{\it invariant factors of $A(t)$}. This can be thought of as a non-Archimedean analogue of the singular value decomposition.
As an application of Theorem \ref{th-amoeba-trop-var-curve} we obtain the following amusing relation between singular values and invariant factors.
\begin{CORO}
Let $A(t)$ be an $n \times n$ matrix whose entries $A_{ij}$ are algebraic functions in $t$. For sufficiently small $t \neq 0$, let $d_1(t) \leq \cdots \leq d_n(t)$ denote the singular values of 
$A(t)$ ordered increasingly. Also let $v_1 \geq \cdots \geq v_n$ be the invariant factors of $A(t)$ ordered decreasingly. We then have:
$$\lim_{t \to 0} (\log_t(d_1(t)), \ldots, \log_t(d_n(t))) = (v_1, \ldots, v_n).$$
\end{CORO}
One can also give a direct proof of this statement using the Hilbert-Courant minimax principle. The authors are not aware of such a statement in the literature relating singular values and invariant factors. 



Generalizing the Bezout theorem, the celebrated Bernstein-Kushnirenko theorem gives a formula for the degree of a hypersurface in the torus $(\k^*)^n$ in terms of the volume of its Newton polytope. This can be translated into a formula for the degree in terms of the tropical variety of the hypersurface. More generally, the intersection numbers of subvarieties in the torus can be computed using the combinatorics of the corresponding tropical varieties/fans. This is the so-called stable intersection of the tropical varieties (\cite{Fulton-Sturmfels}, \cite{Kazarnovskii-fan}). 
The Berstein-Kushnirenko formula has been extended to spherical varieties by Brion and Kazarnovskii (\cite{Brion-degree, Kazarnovskii-group}).
The formula expresses the degree of a hypersurface (which is the divisor of an ample $G$-line bundle) in a spherical variety as the integral of a certain polynomial over the corresponding moment polytope. We expect that in the same fashion the spherical tropical geometry developed in this paper can be extended to give an intersection theory of complete intersections in a spherical variety.



\bigskip

\noindent{\bf Acknowledgement:} We would like to thank Victor Batyrev, Dmitri Timashev, Askold Khovanskii, Gary Kennedy, Jenia Tevelev, Evan Nash, Johannes Hofscheier, Frank Sottile and St\'ephanie Cupit-Foutou
for useful and stimulating discussions and comments. 
We also thank Eric Katz for suggesting to look at the Artin approximation theorem.  

\bigskip

\noindent{\bf Notation:}
Below are some of the notation and conventions used throughout the paper. 
\begin{itemize}
\item $\k$ denotes the base field. It is assumed to be algebraically closed and characteristic $0$, although in some places these assumptions are not necessary.
\item $G$ is a connected reductive algebraic group over $\k$ with a Borel subgroup $B$ and maximal torus $T$. The maximal unipotent subgroup of 
$B$ is denoted by $U$.
\item $\Lambda$ denotes the weight lattice of $T$ with the semigroup of 
dominant weights $\Lambda^+$ corresponding to the choice of $B$. The cone generated by $\Lambda^+$ is the positive Weyl chamber $\Lambda^+_\r$.
\item For a dominant weight $\lambda \in \Lambda^+$ we denote the irreducible $G$-module with highest weight $\lambda$ by $V_\lambda$. We usually denote a highest weight vector in 
$V_\lambda$ by $v_\lambda$.
\item $X$ denotes a spherical $G$-variety over $\k$. 
The ring of regular functions of $X$ is denoted by $\k[X]$.
\item $G/H$ denotes a spherical homogeneous space.
\item For a $G$-variety $X$ the group of weights of $B$-eigenfunctions in $\k(X)$ is denoted by $\Lambda_X$. Also $\Lambda_X^+$ denotes the semigroup of $B$-weights appearing in the algebra of regular functions $\k[X]$. One shows that when $X$ is affine, this semigroup generates the lattice $\Lambda_X$. We denote the set of $G$-invariant valuations on $\k(X)$ by $\V_X$.
\item For a $G$-algebra and domain $A$ the semigroup of highest weights appearing in $A$ is denoted by $\Lambda^+_A$. Also 
$\Lambda_A$ denotes the lattice of $B$-weights appearing in the quotient field of $A$. We also use the notation $\V_A$ to denote the set of $G$-invariant valuations of the quotient field of $A$.
\end{itemize}

\section{Preliminaries on Gr\"obner bases and tropical geometry}
\subsection{Gr\"obner bases} \label{subsec-Grobner-prelim}
The theory of Gr\"obner bases is concerned with ideals in a polynomial ring $\k[x_1, \ldots, x_n]$. The Gr\"obner bases are an
excellent tool in computational algebra and algebraic geometry. One of the many problems that can be answered with an efficient algorithm  
using Gr\"obner bases is the so-called {\it ideal membership problem}: {\it Let $I \subset \k[x_1, \ldots, x_n]$ be an ideal generated by the given polynomials $f_1, \ldots, f_r$. Suppose a polynomial $f$ is given. Determine whether $f$ lies in $I$ or not.}

A {\it term order} is a total order on the semigroup $\z_{\geq 0}^n$ which respects addition.  
We assume the following condition is satisfied:
{\it $(\z_{\geq 0}^n, \succ)$ is maximum well-ordered i.e. any increasing chain has a maximum.}
If the above is satisfied we say that {\it $(\z_{\geq 0}^n, \succ)$ has maximum well-ordered property}. 
This assumption is crucial for the algorithms concerning Gr\"obner bases to terminate. 
{An important example of a term order with the maximum well-ordered property 
is a reversed lexicographic order (corresponding to an ordering of the coordinates in $\z_{\geq 0}^n$).}

Let us fix a term order $\succ$ on $\z_{\geq 0}^n$. Let $f(x) = \sum_\alpha c_\alpha x^\alpha$
be a polynomial in $\k[x_1, \ldots, x_n]$. (Here we have used the multi-index notation, so that $x = (x_1, \ldots, x_n)$ and 
$x^\alpha = x_1^{a_1} \cdots x_n^{a_n}$ where $\alpha = (a_1, \ldots, a_n)$.) One defines the {\it initial monomial of $f$} by:
$$\In_\succ(f) = c_\beta x^\beta,$$ where $\beta = \min\{ \alpha \mid c_\alpha \neq 0 \}$, and the minimum is taken with respect to the term order 
$\succ$. \footnote{Many authors define the initial term using maximum instead of the minimum. Throughout the paper we use minimum convention since it is more compatible with the definition of a valuation in commutative algebra which we will abundantly use in the rest of the paper.} Then given an ideal $I \subset \k[x_1, \ldots, x_n]$ one defines the {initial ideal of $I$} as:
$$\In_\succ(I) = \langle \In_\succ(f) \mid f \in I \rangle.$$
Clearly $\In_\succ(I)$ is a monomial ideal, i.e. it is generated by monomials. The initial ideal $I$ can be described by 
the corresponding collection of lattice points:
$$\{ \alpha \mid x^\alpha \in I \} \subset \z_{\geq 0}^n.$$
It is usually referred to as the {\it staircase diagram of $I$} (with respect to $\succ$).  
This set has a nice combinatorial/geometric description: {\it It is a finite union of shifted copies of $\z_{\geq 0}^n$.}
This is a key observation which is the basis of the theory of Gr\"obner bases.

The following is a basic theorem in the Gr\"obner basis theory (see \cite[Theorem 1.2]{Sturmfels}):
\begin{Th}[Finiteness of number of initial ideals]
Every ideal $I \subset \k[x_1, \ldots, x_n]$ has only finitely many distinct initial ideals (for all possible term orders on $\z_{\geq 0}^n$).
\end{Th}

Fix a term order $\succ$ and let $I$ be an ideal in $\k[x_1, \ldots, x_n]$.
\begin{Def}[Gr\"obner basis] \label{def-Grobner-basis}
A Gr\"obner basis for $I$ (with respect to the term order $\succ$) is a finite set $\G \subset I$ such that the 
set $\In_\succ(\G) = \{ \In_\succ(f) \mid f \in \G \}$ generates $\In_\succ(I)$ as an ideal.
\end{Def}

One shows that if $\G = \{g_1, \ldots, g_r\}$ is a Gr\"obner basis for $I$ then any element $f \in I$ can be written as 
a combination $f = \sum_i h_i g_i$ with $h_i \in \k[x_1, \ldots, x_n]$ using a simple efficient algorithm. The celeberated 
{\it Buchberger algorithm} constructs a Gr\"obner basis for an ideal $I$ from a given set of ideal generators for $I$.

Also given a vector $w \in \r^n$ one can define the initial form of a polynomial with respect to $w$. Let 
$f(x) = \sum_\alpha c_\alpha x^\alpha \in \k[x_1, \ldots, x_n]$ 
then the {\it initial form of $f$ with respect to $w$} is:
$$\In_w(f) = \sum_{w \cdot \beta = m} c_\beta x^\beta,$$
where $m = \min\{ w \cdot \alpha \mid c_\alpha \neq 0\}$. In contrast with the initial monomial with respect to a term order, the initial form 
with respect to a vector $w$ may contain several terms, i.e. is not always a monomial. Although it is not difficult to see that it is a monomial when 
$w$ is in general position. 

Given $w \in \r^n$ and a term order $\succ$ on $\z_{\geq 0}^n$ one defines the term order $\succ_w$ on $\z_{\geq 0}^n$ as follows:
$\alpha \succ_w \beta$ if $w \cdot \alpha > w \cdot \beta$, or $w \cdot \alpha = w \cdot \beta$ and $\alpha \succ \beta$. It is important to notice that this term order has maximum well-ordered property if $w \in \r_{\leq 0}^n$. This property is essential when one wants to have a 
Gr\"obner basis with respect to $\succ_w$. 

Similar to the case of initial ideals with respect to a term order, one can prove that an ideal $I$ has only finitely many distinct initial ideals 
$\In_w(I)$, for $w \in \r^n$. In fact, given an ideal $I$ one can group together the vectors $w$ according to their corresponding initial ideals. Namely, for $w_1, w_2 \in \r^n$ 
we say that $w_1 \sim w_2$ if $\In_{w_1}(I) = \In_{w_2}(I)$. The following theorem is well-known (see \cite[Proposition 2.4]{Sturmfels}):

\begin{Th}[Gr\"obner fan of an ideal]
Let $I$ be a homogeneous ideal. Then the equivalence classes of $\sim$ form a fan in $\r^n$.
\end{Th}

The above fan is usually called the {\it Gr\"obner fan of the ideal $I$}. In Section \ref{sec-spherical-Grobner-fan} we generalize the notion of 
Gr\"obner fan to the context of spherical varieties. 

\subsection{Tropical geometry} \label{subsec-tropical-prelim}
From point of view of algebraic geometry, tropical geometry is concerned with describing the asymptotic behavior, or behavior at infinity, of subvarieties in the algebraic torus $(\k^*)^n$. A subvariety of $(\k^*)^n$ is usually called 
a {\it very affine variety}. The behavior at infinity of a subvariety $Y \subset (\k^*)^n$ is encoded in a piecewise linear object (a fan) 
called the {\it tropical variety} of $Y$.

Since many intersection theoretic data are stable under deformations, the piecewise linear data describing a variety at infinity can often be used to give combinatorial or piecewise linear formula for intersection theoretic problems. Many developments and applications of tropical geometry, at least in algebraic geometry, come from this point of view.

More generally, in tropical geometry instead of varieties in $(\k^*)^n$ one works with varieties in $(K^*)^n$ where $K$ is an algebraically closed field {containing $\k$} 
and equipped with a valuation with values in $\q$, which we denote by $\ord$ (the notation $\ord$ stands for order of vanishing). {We assume that $\ord$ is trivial on $\k$ 
and the residue field of $(K, \ord)$ is $\k$. Also we assume that there is an element $t \in K$ with
$\ord(t)=1$ and and a splitting, i.e. a group homomorphism, $w \mapsto t^w$ from the value group of $\ord$ to $K^*$.}

The main example of such $K$ is the field of formal Puiseux series over $\k$.
Let us recall the definition of the field of formal Puiseux series. Let $\K = \k((t))$ be the field of formal Laurent series in one indeterminate $t$. 
Then $\overline{\K} = \k\{\{t\}\} =  \bigcup_{k=1} \k((t^{1/k}))$ is the field of formal Puiseux series. It is the algebraic closure of the field of Laurent series $\K$. 
The field $\overline{\K}$ has the God-given {\it order of vanishing} valuation 
$\ord_t: \overline{\K} \setminus \{0\} \to \q$ defined as follows: for a Puiseux series $f(t) = \sum_{i=m}^\infty a_i t^{i/k}$, where $a_m \neq 0$, we put $\ord_t(f) = m/k$.

The valuation $\ord$ gives rise to the {\it tropicalization map} $\Trop$ from the torus $(K^*)^n$ to $\q^n$:
$$\Trop(f_1, \ldots, f_n) = (\ord(f_1), \ldots, \ord(f_n)).$$
Let $Y \subset (K^*)^n$ be a subvariety. 

\begin{Def}[Tropicalization] \label{def-tropicalization}
The {\it tropicalization $\Trop(Y)$ of $Y$} is simply defined to be the image of $Y$ under 
the map $\Trop$. We refer to the closure of $\Trop(Y)$ in $\r^n$ as the {\it tropical variety of $Y$}.
\end{Def}

One shows that the tropical variety of a subvariety is always a rational polyhedral complex in $\r^n$. When the variety is defined over $\k$, the tropical variety is 
a rational polyhedral fan in $\r^n$, also referred to as the {\it tropical fan} of the subvariety. 


The notion of a tropical variety can also be defined in terms of the ideal of the variety. We start by defining the notion of a tropical hypersurface.
Let $K[x_1^\pm, \ldots, x_n^\pm]$ denote the algebra of Laurent polynomials with coefficients in $K$ and let $f(x) = \sum_\alpha c_\alpha x^\alpha$ be a Laurent polynomial where 
as usual $x = (x_1, \ldots, x_n)$ and $x^\alpha = x_1^{a_1} \cdots x_n^{a_n}$ with $\alpha = (a_1, \ldots, a_n)$. Let $\supp(f) = \{ \alpha \mid c_\alpha \neq 0 \} \subset \z^n$ 
denote the set of exponents appearing in $f$. 

\begin{Def}[Tropical hypersurface] \label{def-tropical-hypersurface}
The {\it tropical hypersurface of $f$} is the set of all $w \in \r^n$ such that the minimum of $\alpha \mapsto \ord(c_\alpha)+(w \cdot \alpha)$, regarded as a function on $\supp(f)$, 
is attained at least twice.
\end{Def}

More generally one can define the notion of tropical variety of an ideal $I \subset K[x_1^\pm, \ldots, x_n^\pm]$.
First we need a generalization of the notion of initial form of a polynomial from Section \ref{subsec-Grobner-prelim}.
For a vector $w \in \r^n$ we define the initial form $\In_w(f)$ of a Laurent polynomial $f \in K[x_1^\pm, \ldots, x_n^\pm]$ as follows: 
Let $f(x) = \sum_\alpha c_\alpha x^\alpha$ and let $m = \min\{ \ord(c_\alpha)+(w \cdot \alpha) \mid c_\alpha \neq 0\}$. Then the initial form of $f$ is a polynomial 
$\In_w(f) \in \k[x_1^\pm, \ldots, x_n^\pm]$ defined by:
$$\In_w(f)(x) = \sum_{\ord(c_\beta)+(w \cdot \beta) = m} \overline{t^{-\ord(c_\beta)} c_\beta}~ x^\beta,$$ 
where $\overline{a}$ denotes the image of an element $a \in K$ with $\ord(a) \geq 0$ in the residue field $\k$.

The notion of initial ideal $\In_w(I)$ is also defined analogously, i.e. $\In_w(I)$ is the ideal generated by all the $\In_w(f)$, $\forall f \in I$.
\begin{Def}[Tropical variety of an ideal]  \label{def-trop-var-ideal}
The {\it tropical variety of $I$} is the set of all $w \in \r^n$ such that the initial ideal $\In_w(I)$ does not contain any monomials (in other words, 
$\In_w(I) \neq \k[x_1^\pm, \ldots, x_n^\pm]$). 
\end{Def}

The following theorem states that all of the above notions of tropical variety coincide (see \cite{Kapranov} and \cite[Section 3.2]{Maclagan-Sturmfels}).  
\begin{Th}[Fundamental theorem of tropical geometry] \label{th-fundamental-trop-geo}
Let $Y \subset (K^*)^n$ be a subvariety with ideal $I = I(Y) \subset K[x_1^\pm, \ldots, x_n^\pm]$. Then the following sets coincide:
\begin{itemize}
\item[(a)] The intersection of all the tropical hypersurfaces $\trop(f)$, for all $0 \neq f \in I$.
\item[(b)] The closure, in $\r^n$, of the set $\trop(I)$, that is, the set of $w \in \r^n$ such that $\In_w(I)$ contains a monomial.
\item[(c)] The closure, in $\r^n$, of the set $\Trop(Y)$, that is the image of $Y$ under the tropicalization map.
\end{itemize}
\end{Th}

The definition of tropicalization of a subvariety crucially uses the valuation on the field $K$. Hence it is not surprising that tropical geometry is intrinsically related to 
the non-Archimdean analytic geometry. We recall that if $Y$ is an affine variety with coordinate ring $A$ then the {\it Berkovich analytification $Y^{\an}$ of $Y$} is the set of all 
valuations $v: A \to \r \cup \{\infty\}$ equipped with the coarsest topology in which all the maps $v \mapsto v(f)$ are continuous, $\forall f \in A$. The Berkovich 
analytification can be extended to arbitrary varieties $Y$ by gluing the affine pieces. It plays a central role in non-Archimedean geometry (\cite{Gubler}). 

Given an embedding of $Y$ into a torus $(\k^*)^n$ we can define a natural map $Y^{\an} \to \Trop(Y)$. A theorem of Payne states that in fact the analytification $Y^{\an}$
can be realized as the inverse limit of all tropicalizations $\Trop(Y)$ for all possible embeddings of $Y$ (\cite{Payne}).

Finally there is also an Archemidean version of the notion of tropicalization and tropical variety. It is based on the familiar logarithm map on the complex algebraic torus.
Let $\k = \c$ and fix a real number $t > 0$. Consider the logarithm map $\Log_t: (\c^*)^n \to \r^n$ defined by:
$$\Log_t(z_1, \ldots, z_n) = (\log_t|z_1|, \ldots, \log_t|z_n|).$$

\begin{Def}[Amoeba of a subvariety] \label{def-ameoba}
The {\it amobea of $Y$} denoted by $\mathcal{A}_t(Y)$ is the image of $Y$ under the logarithm map $\Log_t$ (clearly it depends on the choice of the base $t$). 
\end{Def}

The following well-known result gives a connection between the amoeba of a subvariety and its tropicalization. 
\begin{Th}[Amoeba approaches the tropical variety]
As $t$ approaches $0$, the amoeba $\mathcal{A}_t(Y)$ converges to the tropical variety of $Y$ (i.e. their Haudorff distance approaches $0$). 
\end{Th}

The book \cite{Maclagan-Sturmfels} is a nice introduction to tropical geometry. The article 
\cite{Shaw} is an introduction to several different topics in tropical geometry.  

In the rest of the paper we discuss generalizations of all of the above definitions and results for subvarieties of a torus $(\k^*)^n$ to subvarieties of a  spherical variety (or a spherical homogeneous space).

\section{Preliminaries on reductive group actions and spherical varieties} \label{sec-spherical}
We start by introducing some notation. As usual let $G$ be a connected reductive algebraic group over an algebraically closed field $\k$. Let $A$ be a  $\k$-algebra and domain with the field of  fractions $K$. Let us assume that $A$ is
a $G$-algebra, that is, $G$ acts on $A$ by algebra isomorphisms. We denote by $K^{(B)}$ the multiplicative group of nonzero $B$-eigenfunctions in the field $K$. 
Also we let $\Lambda_A$  (respectively $\Lambda_A^+$) denote the set of weights which appear as a weight of a $B$-eigenfunction in $K$ 
(respectively in the algebra $A$).

Clearly $\Lambda_A$ is a sublattice of the weight lattice $\Lambda$ of $G$. 
One shows that if $A$ is a finitely generated algebra then $\Lambda_A^+$ is a finitely generated semigroup.
Moreover, the semigroup $\Lambda_A^+$ generates the lattice $\Lambda_A$ (see for example \cite[Propositions 5.5 and 5.6]{Timashev}).

We are interested in the case when $A = \k[X]$ is the algebra of regular functions on a $G$-variety $X$.
Then $K = \k(X)$ is the field of rational functions on $X$. In this case, we will denote the lattice $\Lambda_A$ and the semigroup $\Lambda_A^+$ 
by $\Lambda_X$ and $\Lambda_X^+$ respectively.


A normal $G$-variety $X$ is called {\it spherical}
if $B$ has a dense open orbit (note that since all the Borel subgroups are conjugate this is independent of the choice of $B$).
A homogeneous space ${G/H}$ is spherical if it is spherical for the left action of $G$. 

Since $B$ has an open orbit it follows that the map which assigns to a $B$-eigenfunction its weight, gives an 
isomorphism between $\k(X)^{(B)} / \k^*$ and $\Lambda_X$. One can show that there is a natural choice of a 
torus $T_X \in G$ such that the weight lattice of $T_X$ can be identified with the lattice $\Lambda_X$.

It can be shown that if $X$ is a quasi-affine spherical $G$-variety then $X$ is {\it strongly quasi-affine}. This means that 
the algebra of regular functions $A=\k[X]$ is finitely generated and the natural map $X \to \Spec(A)$ is an open embedding 
(see for example \cite[Proposition 2.2.3]{Yiannis}). The affine variety $X_{\aff} = \Spec(A)$ is usually called the {\it affine closure of $X$}.

For the rest of the paper we work with a spherical $G$-variety $X$. We also work with the algebra of regular functions $A=\k[X]$, in which case we assume $X$ to be quasi-affine.

\begin{Ex} \label{ex-sph-var}
Here are some examples of spherical varieties. We will use (3) below as a simple example to illustrate several concepts and results in the paper (see Section \ref{subsec-punctured-plane}). 
\begin{itemize}
\item[(1)] Let $G = T = (\k^*)^n$ be an algebraic torus. Then a spherical $T$-variety is the same a toric $T$-variety. 
\item[(2)] Let $X = G/P$ be a parital flag variety. By the Bruhat decomposition it is a spherical variety for the left action of $G$.
It is an example of a projective spherical variety.
\item[(3)] Consider the natural linear action of  $G = \SL(2, \k)$ act on $\A^2 \setminus \{(0,0)\}$. It is easy to see that $G$ acts transitively on 
$\A^2 \setminus \{(0,0)\}$. The stabilizer of the point $(1, 0)$ is the maximal unipotent subgroup $U$ of upper triangular matrices with $1$'s on the diagonal.
Thus $\A^2 \setminus \{(0,0)\}$ can be identified with the homogeneous space $G/U$. Let $B$ be the subgroup of upper triangular 
matrices. The $B$-orbit of the point $(0,1)$ is the open subset $\{(x, y) \mid y \neq 0\}$. Hence $\A^2 \setminus \{(0,0)\}$ is a spherical $\SL(2, \k)$-variety. 
It is an example of a quasi-affine spherical variety. Similarly, one verifies that $\A^n \setminus \{0\}$ is a spherical variety for the natural action of $G = \SL(n, \k)$ 
(note that for $n>2$, the $\SL(n, \k)$-stabilizer of a point in $\A^n \setminus \{0\}$ is larger than a maximal unipotent subgroup).
\item[(4)] More generally, let $X = G/U$ where $U$ is a maximal unipotent subgroup of $G$. Again by the Bruhat decomposition $X$ is a spherical $G$-variety for 
the left action of $G$. It is well-known that $X$ is a quasi-affine variety. Note that there is a natural projection from $G/U$ to $G/B$, where $B$ is the Borel subgroup containing $U$.
\item[(5)] Let $X = G$ and consider the left-right action of $G \times G$ on $X=G$. Note that this action is transitive and the stabilizer of the identity $e$ is the subgroup 
$G_{\textup{diag}} = \{(g, g) \mid g \in G\}$. Thus $G$ can be identified with the homogeneous space $(G \times G) / G_{\textup{diag}}$. 
Again from the Bruhat decomposition it follows that $X=G$ is a $(G \times G)$-spherical variety. The $(G \times G)$-equivariant completions 
of $G$ are usually called {\it group compactifications}.
\end{itemize}
\end{Ex}

\subsection{Invariant valuations and spherical roots} \label{subsec-inv-valuation}

As usual let $X$ be a spherical $G$-variety. In this section we consider the set of $G$-invariant valuations on the field of rational functions $\k(X)$.

Let $v$ be a valuation on the field $\k(X)$ with values in $\q$. By restriction the valuation $v$ gives a linear map on $\k(X)^{(B)}$ and hence on the lattice 
$\Lambda_X$. Let $\rho(v) \in \Hom(\Lambda_X, \q)$ denote this linear map. We will denote the dual space $\Hom(\Lambda_X, \q)$ by by $\mathcal{Q}_X$.

\begin{Def}[$G$-invariant valuation] \label{def-G-inv-val}
A valuation $v$ on $\k(X)$ is $G$-invariant if for any $g \in G$ and $f \in \k(X)$ we have $v(f) = v(g \cdot f)$. 
We denote the set of all $G$-invariant valuations with values 
in $\q$ by $\V_X$. 
\end{Def}

The following is well-known (\cite{Luna-Vust}):
\begin{Th}
\begin{itemize}
\item[(a)] The map $\rho: \V_X \to \mathcal{Q}_X$ is one-to-one, that is, a $G$-invariant valuation is determined by its restriction on the $B$-eigenfunctions.
\item[(b)] The image of $\rho$ is a convex polyhedral cone in the vector space $\mathcal{Q}_X$. We will identify $\V_X$ with its image under $\rho$.
\item[(c)] Let $C_X$ denote the image of the anti-dominant Weyl chamber in $\mathcal{Q}_X$. Then $\V_X$ 
contains $C_X$.
\end{itemize}
\end{Th}

We note that the value of a
$G$-invariant valuation $v$ on a $B$-weight vector $f_\lambda$ (with weight $\lambda$) depends only on $v$ and $\lambda$. 
We denote this value by $\langle v, \lambda \rangle$. This is in fact the natural pairing between the vector $\rho(v) \in \mathcal{Q}_X$ and $\lambda \in \Lambda_X$. 

\begin{Rem} \label{rem-val-weight-vec}
In the case when $G=T=(\k^*)^n$ and $A = \k[x_1^\pm, \ldots, x_n^\pm]$ is the algebra of Laurent polynomials, the pairing between 
invariant valuations and weights is the usual dot product. A vector $w \in \q^n$ gives rise to a so-called 
{\it weight valuation} $v_w$ on $A$. If $f(x) = \sum_\alpha c_\alpha x^\alpha$ then $v_w(f)$ is by definition given by:
$$v_w(f) = \min\{ w \cdot \alpha \mid c_\alpha \neq 0 \}.$$
The value of $v_w$ on a monomial $x^\alpha$ (i.e. a weight vector with weight $\alpha$) is given by the dot product $w \cdot \alpha$.  
\end{Rem}

The following theorem is due to Brion (\cite{Brion}) and Knop (\cite{Knop}).
\begin{Th} \label{th-val-cone-simplicial}
The set $\V_X$ is a simplicial cone in the vector space $\mathcal{Q}_X$. Moreover, it is the fundamental domain for the Weyl 
group of a root system. 
More precisely, one can choose a set of simple roots $\beta_1, \ldots, \beta_\ell$ in this root system such that the cone $\V_X$ is defined by:
\begin{equation} \label{equ-valuation-cone-inequ}
\V_X = \{ v \in \mathcal{Q}_X \mid \langle v, \beta_i \rangle \leq 0 , ~\forall i=1, \ldots, \ell \}.
\end{equation}
\end{Th}

\begin{Def}[Spherical roots]   \label{def-sph-roots}
The set of simple roots $\{\beta_1, \ldots, \beta_\ell\}$ is called the system of {\it spherical roots} of $X$.
This Weyl group of the spherical root system is called the {\it little Weyl group of $X$}.
\end{Def}

\begin{Rem} \label{rem-tail-cone-sph-roots}
By Theorem \ref{th-coor-ring-multiplication-spherical}, the set of spherical roots generates the so-called {\it tail cone} of the spherical variety $X$.
\end{Rem}

\begin{Rem} \label{rem-Knop-val-cone}
The above theorem (Theorem \ref{th-val-cone-simplicial}) can be extended to arbitrary $G$-varieties. Let $X$ be a (not necessarily spherical) $G$-variety. Consider the subfield of $B$-invariant rational functions $\k(X)^B$. Fix a $\q$-valued valuation $v_0$ on the field $\k(X)^B$. Let $\V_0$ denote the collection of $G$-invariant valuations on $\k(X)$ whose restriction on $\k(X)^B$ 
coincides with $v_0$. Similar to the above, we can define a map $\rho_0$ from $\V_0$ to $\Hom_\z(\Lambda, \q)$ as follows. Fix a $G$-invariant valuation $v_1$ in $\V_0$. 
Then for $v \in \V_0$ define $\rho_0(v)(\lambda) = v(f_\lambda) - v_1(f_\lambda),$ where $f_\lambda \in \k(X)$ denotes a $B$-eigenfunction with weight $\lambda$. 
One verifies that this is well-defined i.e. 
$\rho_0(v)(\lambda)$ is independent of the choice of the $B$-eigenfunction $f_\lambda$. Then the set $\V_0$ can be identified with a simplicial cone in the vector space 
$\Hom(\Lambda_X, \q)$ which is the fundamental domain for a root system.
\end{Rem}

Now we assume that $X$ is a quasi-affine spherical variety and consider its ring of regular functions $A = \k[X]$.

As a $G$-module we can decompose $A$ into:
$$A = \k[X] = \bigoplus_{\lambda \in \Lambda_X^+} W_\lambda,$$
where $W_\lambda$ is the $\lambda$-isotypic component in the $G$-module $A$. Since $X$ is spherical 
the algebra $A$ is multiplicity-free and hence $W_\lambda \cong V_\lambda$ or $\{0\}$.

The following proposition is an important observation about invariant valuations. We record it here for later reference.
\begin{Prop} \label{prop-inv-val-min-isotypic}
Let $v \in \V_X$ be a $G$-invariant valuation. For each $\lambda \in \Lambda_X^+$ let $h_\lambda$ be a $B$-eigenvector in $W_\lambda$ 
(it is unique up to a scalar).
\begin{itemize}
\item[(1)] For any $0 \neq f_\lambda \in W_\lambda$ we have $v(f_\lambda) = v(h_\lambda)$.
\item[(2)] Let $f \in A$ and write $f=\sum_\lambda f_\lambda$ as the sum of its
isotypic components. We then have:
$$v(f) = \min\{ v(h_\lambda) \mid f_\lambda \neq 0\}.$$
\end{itemize}
\end{Prop}
\begin{proof}
(1) First we note that the $G$-module $W_\lambda$ is spanned by the $G$-orbit of the highest weight vector $h_\lambda$. 
Thus we can find $g_1, \ldots, g_s \in G$ such that $f_\lambda$ is a linear combination of the $g_i \cdot h_\lambda$. From the $G$-invariance and 
non-Archimedean property of $v$ it then follows that $v(f_\lambda) \geq v(h_\lambda)$. But since $W_\lambda$ is also spanned 
by the $G$-orbit of $f_\lambda$. Reversing the roles of $f_\lambda$ and $h_\lambda$ in the above argument we see that $v(f_\lambda) = v(h_\lambda)$.
(2) By the non-Archimedean property of $v$ we have $v(f) \geq \min\{ v(f_\lambda) \mid f_\lambda \neq 0\}$. Note that by (1) above, 
the righthand side is equal to $\min\{ v(h_\lambda) \mid f_\lambda \neq 0\}$. To prove the reverse inequality let $M_f$ be the $G$-submodule of $A$ generated by $f$. 
By an argument similar to that in (1) we see that for every nonzero $f' \in M_f$ we have $v(f') \geq v(f)$. But all the isotypic components $f_\lambda$ of $f$ lie in $M_f$ 
and hence $v(f) \leq \min\{ v(f_\lambda) \mid f_\lambda \neq 0\}$. This finishes the proof.
\end{proof}

Finally, we discuss the multiplication in the $G$-algebra $A$.
Let us remind the definition of the {\it dominant order} on the weight lattice $\Lambda$. For two dominant weights $\lambda$, $\mu$ we say $\lambda \geq \mu$ if $\mu - \lambda$ is 
a linear combination of the simple roots with nonnegative integer coefficients. (Note that what we have defined is the reverse of the usual dominant order in the literature. 
We are using this convention to be consistent with the definition of a valuation.)
The dominant order has the important property that: 
{\it for $\lambda, \mu, \gamma \in \Lambda^+$, if $V_\gamma$ appears in 
$V_\lambda \otimes V_\mu$ then $\gamma \geq \lambda + \mu$.} 
From the above statement regarding the irreducible $G$-modules appearing in the tensor product one can conclude the following.

\begin{Th} \label{th-coor-ring-multiplication-G-algebra}
Let $R$ be any $G$-algebra and let us write $R = \bigoplus_\lambda W_\lambda$ where $W_\lambda$ is the $\lambda$-isotypic component in $R$ for $\lambda \in \Lambda$.
Let $f \in W_\lambda$ and $g \in W_\mu$ then $fg$ lies in: 
$$\bigoplus_{\gamma \geq \lambda + \mu} W_\gamma.$$
\end{Th}

We can also define an analogue of the dominant order for the sublattice $\Lambda_X$ associated to the spherical variety $X$.  
Let $\lambda, \mu \in \Lambda_X$. We say that $\lambda \geq_X \mu$ if $\mu - \lambda$ is 
a linear combination of the spherical roots with nonnegative integer coefficients. 
We call $>_X$ the {\it spherical dominant order}.
We have the following refinement of Theorem \ref{th-coor-ring-multiplication-G-algebra} in the spherical case (\cite[Section 5]{Knop-LV}):
\begin{Th} \label{th-coor-ring-multiplication-spherical}
Let $A = \k[X] = \bigoplus_{\lambda \in \Lambda_X^+} W_\lambda$ be the ring of regular functions on $X$.
Let $f \in W_\lambda$, $g \in W_\mu$. Then $fg$ lies in: 
$$\bigoplus_{\gamma \geq_X \lambda + \mu} W_\gamma.$$
\end{Th}

Take $\lambda, \mu \in \Lambda_X^+$ and let $W_\nu$ appear in the product $W_\lambda W_\mu$. A weight of the form $\lambda + \mu - \nu$ is usually called a {\it tail} and the closure of the cone 
in $\Lambda_X \otimes \r$ generated by all the tails is called the {\it tail cone of $X$}. Theorem \ref{th-coor-ring-multiplication-spherical} implies the following (see for example \cite[Lemma 5.1]{Knop-LV}):    
\begin{Cor} \label{cor-tail-cone-val-cone}
The tail cone is the dual cone to $-\V_X$, where as usual $\V_X$ is the valuation cone. 
\end{Cor}


\subsection{Horospherical contraction} \label{subsec-horosph-contraction}
We recall the notion of horospherical contraction of a $G$-algebra (\cite{Popov}). Let $R = \bigoplus_\lambda W_\lambda$ be a (rational) $G$-algebra. 
Consider the $\Lambda^+$-filtration $F_\bullet$ on $R$ defined as follows. For
$\lambda \in \Lambda^+$ let
$$F_{\geq \lambda} = \bigoplus_{\mu \geq \lambda} W_\mu.$$ From Theorem \ref{th-coor-ring-multiplication-G-algebra} it follows that $F_\bullet$ is a multiplicative filtration. 
Let $$R_{\hc} = \gr_{F_\bullet} (R) = \bigoplus_\lambda F_{\geq \lambda} / F_{> \lambda},$$ denote the associated graded algebra of $R$. The $\Lambda^+$-graded algebra $R_{\hc}$
is usually called the {\it horospherical contraction of $R$}. The algebra $R_{\hc}$ is isomorphic to $R$ as a $G$-algebra but it has a different (simpler) multiplication.
If $R$ is the coordinate ring of an affine $G$-variety $X$ then the {\it horospherical contraction of $X$} is the variety $\Spec(R_{\hc})$. The horospherical contraction of $X$ is indeed a horospherical $G$-variety. We recall 
that a $G$-variety is called {\it horospherical} if $G$-stabilizer of any point contains a maximal unipotent subgroup. 

Now let $A = \bigoplus_{\lambda \in \Lambda_X^+} W_\lambda$ be the ring of regular functions on a quasi-affine spherical variety $X$. 
Then $A$ is a multiplicity-free $G$-module. As a $G$-module 
the horospherical contraction $A_{\hc}$ is also $\bigoplus_\lambda W_\lambda$, but 
for any $\lambda, \mu \in \Lambda_X$ the multiplication map is given by a Cartan multiplication 
$W_\lambda \times W_\mu \to W_{\lambda + \mu}$. 

We also note that, by Theorem \ref{th-coor-ring-multiplication-spherical}, if instead of the dominant oder $>$ we use the weaker order $>_X$ on $\Lambda_X$ to define 
the filtration $F_\bullet$, the resulting associated graded $A_{\hc} = \gr_{F_\bullet}(A)$ is the same.

Finally we define the notion of a $\Lambda_X$-homogeneous ideal in the horospherical contraction $A_{\hc}$. 
It is a generalization of the notion of a monomial ideal in a polynomial ring. 
\begin{Def}[Homogeneous ideal] \label{def-Lambda-A-homog-ideal}
We call an ideal $J \subset A_{\hc}$ a $\Lambda_X^+$-homogeneous ideal if it is generated by a finite number of $\Lambda_X^+$-homogeneous elements. 
That is, we can find a set of generators $f_1, \ldots, f_s$ for $J$ such that each $f_i$ lies in some $W_{\lambda_i}$, $\lambda_i \in \Lambda_X^+$. 
\end{Def}

The next proposition is straightforward to verify.
\begin{Prop} \label{prop-Lambda-A-homog-ideal}
An ideal $J \subset A_{\hc}$ is a $\Lambda_X$-homogeneous ideal if and only if the following holds: Let $f \in J$ with isotypic decomposition $f = \sum_\lambda f_\lambda$.
Then $f_\lambda \in J$ for all $\lambda$.
\end{Prop}

\section{Gr\"obner theory for multiplicity-free $G$-algebras} \label{sec-Grobner-G-alg}
\subsection{Spherical Gr\"obner bases} \label{subsec-sph-Grobner}
As usual let $A=\k[X]$ be the algebra of regular functions on a quasi-affine spherical $G$-variety $X$
(alternatively we can take $A$ to be a finitely generated $G$-algebra which is a multiplicity-free rational $G$-module). 
In this section we develop basics of a Gr\"obner theory for ideals in the $G$-algebra $A$.
This will be used in the next sections where we define and explore the notion of a spherical tropical variety for a subscheme in a spherical homogeneous space.


First we define the notion of initial ideal with respect to a total order. 
Let $\succ$ be a total order on the weight lattice $\Lambda_X$ respecting addition. 

\begin{Ass} \label{ass-total-order}
We will always make the following assumptions on the total order $\succ$.
\begin{itemize}
\item[(1)] The total order $\succ$ refines the spherical dominant order $>_X$ (which in general is a partial order). That is, for two weights $\lambda$, $\mu$ if we have 
$\lambda >_X \mu$ then $\lambda \succ \mu$.
\item[(2)] The semigroup $\Lambda_X^+$ is maximum well-ordered with respect to $\succ$ i.e. any increasing chain has a maximum element. 
\footnote{In the usual Gr\"obner theory literature (over a polynomial 
ring) it is customary to assume that the total order (term order) is minimum well-ordered. In this paper we use the minimum convention, i.e. we define the initial term
using minimum, in order to be compatible with the usual definition of a valuation. 
That is why we need the maximum well-ordered property, as opposed to the minimum well-ordered property.    
}
\end{itemize}
\end{Ass}


\begin{Rem} \label{rem-max-well-ordered-ass}
The assumption (2) above is needed to guarantee that several algorithms regarding spherical Gr\"obner bases terminate. For example 
this assumption is essential in Propositions \ref{prop-sph-Grobner-basis-generate-ideal} and \ref{prop-sph-div-algo} below. 
\end{Rem}

\begin{Rem}  \label{rem-good-ordering-exists}
If $A$ is a positively graded $G$-algebra and $\dim(A_i) < \infty$, for all $i$, then $A$, regarded as a $(\k^* \times G)$-algebra admits an ordering $\succ$ satisfying Assumption \ref{ass-total-order}.
\end{Rem}

The total order $\succ$ gives rise to a 
filtration on $A$. Namely, for each $\lambda \in \Lambda_X$ we define:
$$A_{\succeq \lambda} = \bigoplus_{\mu \succeq \lambda} W_\mu.$$
The space $A_{\succ \lambda}$ is defined similarly. We denote the associated graded of this filtration by $\gr_\succ(A)$, that is:
$$\gr_\succ(A) = \bigoplus_{\lambda \in \Lambda_X^+} A_{\succeq \lambda} / A_{\succ \lambda}.$$

We have a natural $G$-module isomorphism between 
$A$ and $\gr_\succ(A)$ defined as follows. For each $\lambda \in \Lambda_X^+$ and $f_\lambda \in W_\lambda \subset A$ send $f_\lambda$ to its image in the 
quotient space $A_{\succeq \lambda} / A_{\succ \lambda} \subset \gr_\succ(A)$. It can be verified that this map extends to give a $G$-module isomorphism between $A$ and $\gr_\succ(A)$. 

Moreover, we have the following.
\begin{Prop}[Associated graded of a total order]
For any total order $\succ$ on $\Lambda_X$ as above, the associated graded $\gr_\succ(A)$ is naturally isomorphic, as a $G$-algebra, to the horospherical 
contraction $A_{\hc}$.
\end{Prop}
\begin{proof}
Follows from the assumption that $\succ$ refines the spherical dominant order and Theorem \ref{th-coor-ring-multiplication-spherical}.
\end{proof}

For $f \in A$ let us write $f = \sum_{\lambda} f_\lambda$ with $f_\lambda \in W_\lambda$. Let $\mu = \min\{ \lambda \mid f_\lambda \neq 0\}$ be the smallest dominant weight appearing in $f$. 
Here the minimum is with respect to the total oder $\succ$. Clearly $f \in A_{\succeq \mu}$. We define $\In_\succ(f)$ to be the image of $f$ in 
$A_{\succeq \mu} / A_{\succ \mu} \subset \gr_{\succ}(A)$. We call $\In_\succ(f)$ the {\it initial term of $f$ with respect to $\succ$}. 

\begin{Def}[Initial ideal with respect to a total order] \label{def-initial-ideal->}
Let $I \subset A$ be an ideal. We denote by $\In_\succ(I)$ the ideal in $\gr_\succ(A)$ generated by 
all the initial terms $\In_\succ(f)$ for $f \in I$. 
\end{Def}

\begin{Def}[Spherical Gr\"obner basis] \label{def-sph-Grobner-basis}
If $\G \subset I$ is such that $\{\In_\succ(f) \mid f \in \G \}$ generates the initial ideal $\In_\succ(I) \in \gr_{\succ}(A)$ then we call $\G$ a 
{\it spherical Gr\"obner basis for $I$ with respect to the total order $\succ$}. 
\end{Def}


As in the usual Gr\"obner basis theory we have the following.
\begin{Prop} \label{prop-sph-Grobner-basis-generate-ideal}
Let $\G$ be a spherical Gr\"obner basis for $I$ with respect to $\succ$ then $\G$ generates $I$ as an ideal.
\end{Prop}
\begin{proof}
Let $0 \neq h \in I$ and suppose for $\lambda_0 \in \Lambda_X^+$ we have $h \in A_{\succeq \lambda_0}$ but $h \notin A_{\succ \lambda_0}$. 
{Since $\In(\G) = \{\In(f) \mid f \in \G\}$ generates the initial ideal $\In_\succ(I)$ we can find $f_1, \ldots, f_s \in \G$ and 
$h_1, \ldots, h_s \in A$ such that $\In(h) = \In(\sum_i h_if_i)$.}
This means that $h - \sum_i h_if_i$ lies in the subspace $A_{\succ \lambda_0}$. If $h_1 = h - \sum_i h_if_i$ is nonzero we find $\lambda_1$ such that 
$h_1 \in A_{\succeq \lambda_1}$ but $h_1 \notin A_{\succ \lambda_1}$ and continue. Thus we get a sequence of elements 
$\lambda_0 \precneqq \lambda_1 \precneqq \lambda_2 \precneqq \cdots $. By the maximum well-ordering assumption we cannot have a strictly 
increasing chain. This means that at some stage we should arrive at $0$ which implies that $h$ is in the ideal generated by $\G$.
\end{proof}

We can also define an analogue of the reduced Gr\"obner basis.
\begin{Def}[Reduced spherical Gr\"obner basis] \label{def-reduced-sph-Grobner-basis}
With notation as above, let $\G$ be a spherical Gr\"obner basis for an ideal $I \subset A$.
We call $\G$ a {\it reduced spherical Gr\"obner basis}  if the following holds: for every $f \in \G$ write $f = \sum_\lambda f_\lambda$ as sum of its isotypic components. 
Let $\lambda_0$ be the minimum (with respect to $\succ$) among $\{ \lambda \mid f_\lambda \neq 0\}$. We then require that for every nonzero $f_\lambda$ with 
$\lambda \neq \lambda_0$, $\In_\succ(f_\lambda)$ does not lie in the initial ideal $\In_\succ(I) = \langle \In_\succ(f) \mid f \in \G \rangle$.
\end{Def}
Using a similar argument as in the usual Gr\"obner theory one shows that any ideal has a reduced spherical Gr\"obner basis. 

We also have a version of the division algorithm. The proof is analogous to the usual division algorithm in Gr\"obner theory. 
\begin{Prop}[Spherical division algorithm] \label{prop-sph-div-algo}
Let $I \subset A$ be an ideal.
Then any $f \in A$ can be written as: 
$$f = h + \sum_\lambda f_\lambda,$$ where $h \in I$ and for each $\lambda$, $0 \neq f_\lambda \in W_\lambda$ does not lie in $\In_\succ(I)$. 
Moreover $\sum_\lambda f_\lambda = 0$ if and only if $f \in I$.
\end{Prop}
\begin{proof}
Proceed as in the proof of Proposition \ref{prop-sph-Grobner-basis-generate-ideal}.
\end{proof}


\begin{Rem} \label{rem-division-algo-involved}
The spherical division algorithm is more complicated to implement than the usual division algorithm in a polynomial ring. 
This is because, in the general spherical setting we have two new ingredients involving the multiplication in our algebra $A$: 
\begin{itemize}
\item For $\lambda, \mu \in \Lambda_X^+$,
the multiplication sends $W_\lambda \times W_\mu$ to 
$W_{\lambda + \mu} \oplus \bigoplus_{\gamma \geq_X \lambda+\mu } W_\gamma$, as opposed to the case of polynomials where the product of two monomials is just another 
monomial (see Theorem \ref{th-coor-ring-multiplication-spherical}). 
\item Even computing the leading component of the multiplication i.e. the map $W_\lambda \times W_\mu \to W_{\lambda + \mu}$ obtained by projection onto the $W_{\lambda + \mu}$ 
component, involves some more computation (this is in fact a Cartan multiplication and describes the multiplication in the horospherical contraction $A_{\hc}$).
\end{itemize}
Similarly the Buchberger algorithm is more involved. 
Nevertheless the authors beleive that one can introduce a ``nice'' spherical Gr\"obner basis which would make spherical 
division algorithm and the Buchberger algorithm more effective (using canonical bases from representation theory). 
\end{Rem}

To illustrate the concepts we give a baby example below. 
\begin{Ex} \label{ex-SL(2)-Grobner-basis}
As in Example \ref{ex-sph-var}(3) let $G = \SL(2, \k)$ act on $X=\A^2$ in the natural way. The algebra of regular functions 
$A = \k[X]$ is just the ring of polynomials $\k[x, y]$. The weight lattice $\Lambda_X$ is $\z$ with the semigroup $\Lambda_X^+ = \z_{\geq 0}$. We note that in this case the horospherical
contraction $A_{\hc}$ is just $A$ itself. 
As the total order $\succ$ we take the reverse of the natural ordering on $\z$. Then $(\z_{\geq 0}, \succ)$ is maximum well-ordered. The associated graded $\gr_\succ(A)$ is then 
naturally isomorphic to $A$. We have the isotypic decomposition:
$$A = \bigoplus_{d = 0}^\infty A_d,$$ where $A_d = \k[x, y]_d$ is the vector space of homogeneous polynomials of degree $d$. 
If $f \in \k[x,y]$ is a polynomial of degree $d$ then $\In_\succ(f)$ is just $f_d$, the sum of monomials in $f$ of degree $d$. 
Now let us explain how to get a spherical Gr\"obner basis for an ideal in $\k[x, y]$.
Let $I \subset \k[x, y]$ be an ideal. Let $>$ denote some lexicographic order on $x$ and $y$ e.g. $x > y$. Also let $w = (1, 1)$ with $>_w$ the corresponding total order on $\z_{\geq 0}^2$.
That is, $(a_1, a_2) >_w (b_1, b_2)$ if $a_1+a_2 > b_1+b_2$, or $a_1+a_2 = b_1+b_2$ and $(a_1, a_2) > (b_1, b_2)$. One verifies that a Gr\"obner basis (in the usual sense) 
for $I$ with respect to $>_w$ is also a spherical Gr\"obner basis for $I$ with respect to $\succ$.  
\end{Ex}

Similarly to the usual Gr\"obner theory we can prove the key statement that an ideal has only finitely many initial ideals.
\begin{Th} \label{th-initial-ideal-finite}
Every ideal $I$ has a finite number of initial ideals, where for each total order $\succ$ we consider the initial ideal $\In_\succ(I)$ as an ideal in 
the algebra $A_{\hc}$ via the natural isomorphism $\gr_\succ(A) \cong A_{\hc}$. 
\end{Th}


\begin{proof}[Proof of Theorem \ref{th-initial-ideal-finite}]
We adapt the proof in \cite[Lemma 2.6]{Mora-Robbiano} as well as \cite[Theorem 1.2]{Sturmfels}  to our situation. 
By contradiction suppose the set $\Sigma_0$ of all the distinct initial ideals of $I$ is infinite. 
Choose a nonzero $f_1 \in I$ and write $f_1 = \sum_\lambda f_{1, \lambda}$ with $0 \neq f_{1, \lambda} \in W_\lambda$. Among the $f_{1, \lambda}$ we can then find a 
$0 \neq f_{1, \lambda_1}$ such that the set $\Sigma_1 = \{ \In_\succ(I) \mid f_{1, \lambda_1} \in \In_\succ(I) \}$ is infinite. Consider the ideal 
$J_1 = \langle f_{1, \lambda_1} \rangle \subset A_{\hc}$. Since there are infinitely many distinct ideals in $\Sigma_1$ clearly one of them strictly contains $J_1$.  
Let $\succ$ be such that $\In_\succ(I)$ strictly contains $J_1$. 
Since both of these ideals are $\Lambda_X$-homogeneous there is a homogeneous element 
$f_\mu \in \In_\succ(I) \setminus J_1$. From the definition of an initial ideal we can find $f_2 \in I$ with $\In_\succ(f_2) = f_\mu$. 
Let us write $f_2 = \sum_\lambda f_{2, \lambda}$. {By repeatedly eliminating the components that lie in $J_1$  (as in the proof the spherical division algorithm), 
we can assume that none of the $0 \neq f_{2, \lambda}$ lies in $J_1$.} 
Thus we can find a nonzero $f_2 = \sum_\lambda f_{2, \lambda} \in I$  
with the property that none of the $0 \neq f_{2, \lambda}$ lies in $J_1$.
Then among the $f_{2, \lambda}$ there exists $f_{2, \lambda_2}$ such that the set $\Sigma_2 = \{ \In_\succ(I) \in \Sigma_1 \mid f_{2, \lambda_2} \in \In_\succ(I) \}$ is 
infinite. Next let $J_2 = \langle f_{1, \lambda_1}, f_{2, \lambda_2} \rangle \subset A_{\hc}$. Again there exists an initial ideal in $\Sigma_2$ which strictly contains $J_2$. Repeating 
the above argument we see that there exists $0 \neq f_3 \in I$ such that none of its components lies in $J_2$ and so on. Continuing we arrive at an increasing chain of 
ideals $J_1 \subsetneqq J_2 \subsetneqq \cdots$ in $A_{\hc}$. This contradicts that $A_{\hc}$ is Noetherian and the theorem is proved.
\end{proof}

\begin{Def}[Universal spherical Gr\"obner basis] \label{def-universal-Grobner-basis}
Let $I$ be an ideal in $A$. We say that $\G \subset A$ is a {\it universal spherical Gr\"obner basis} for $I$
if for any total order $\succ$ the set $\{\In_\succ(f) \mid f \in \G\}$ generates the initial ideal $\In_\succ(I) \in \gr_\succ(A)$. 
\end{Def}

\begin{Cor}[Existence of a finite universal spherical Gr\"obner basis] \label{cor-universal-Grobner-basis}
There exists a finite universal spherical G\"obner basis.
\end{Cor}
\begin{proof}
The corollary follows immediately from the finiteness of the number of initial ideals (Theorem \ref{th-initial-ideal-finite}).
\end{proof}

\subsection{Partial horospherical contractions associated to faces of valuation cone} \label{subsec-partial-horo}
As usual $X$ is a quasi-affine spherical $G$-variety with $A=\k[X]$. 
Let $v \in \V_X$ be a $G$-invariant valuation. The valuation $v$ gives rise to a filtration on the algebra $A$ defined as follows. For every $a \in \q$ put:
$$A_{v \geq a} = \{ f \in A \mid v(f) \geq a \}.$$
We note that:
\begin{equation} \label{equ-A_v}
A_{v \geq a} = \bigoplus_{\langle v, \gamma \rangle \geq a} W_\gamma.
\end{equation}
The subspace $A_{v > a}$ is defined similarly.
The associated graded algebra of $v$ is defined to be:
$$\gr_v(A) = \bigoplus_{a \in \q} A_{v \geq a} / A_{v > a}.$$
Since $v$ is $G$-invariant, each subspace in the filtration is $G$-stable and the algebra $\gr_v(A)$ is naturally a $G$-algebra.
For each $f$ with $v(f) = a$ let $\In_v(f)$ denote the image of $f$ in the quotient space 
$A_{v \geq a} / A_{v > a}$. 

Below we show that the graded algebra $\gr_v(A)$ only depends on the face of the valuation cone on which $v$ lies.

As before let $\beta_1, \ldots, \beta_\ell \in \Lambda_X$ denote the simple spherical roots for $X$. Since the valuation cone $\V_X$ is 
simplicial, there is a one-to-one correspondence between the subsets of the simple roots and the faces of $\V_X$. A subset $S \subset 
\{\beta_1, \ldots, \beta_\ell\}$ determines a face $\sigma$ by:
\begin{equation} \label{equ-face-sigma}
\sigma = \{ v \in \V_X \mid \langle v, \beta \rangle = 0, \forall \beta \in S \}.
\end{equation}
\begin{Def} \label{def-partial-order-corr-face}
To a face $\sigma$ of the valuation cone $\V_X$ we can also associate a partial order $>_\sigma$ which is weaker than the spherical dominant 
order $>_X$. For $\lambda, \mu \in \Lambda_X$ we say that $\lambda >_\sigma \mu$ if $\mu - \lambda = \sum_{i} c_i \beta_i$ where the $c_i$ are nonnegative integers 
and at least for one $\beta_i \notin S$ we have $c_i \neq 0$.
\end{Def}

The partial order $>_\sigma$ in turn gives rise to a partial horospherical contraction of $A$. More precisely, let 
$F_{\sigma, \bullet}$ be the $\Lambda_X^+$-filtration on $A$ defined as follows. For $\lambda \in \Lambda_X^+$ put:
$$F_{\sigma, \lambda} = \bigoplus_{\gamma \geq_\sigma \lambda} W_\gamma.$$ We denote the associated graded of the filtration 
$F_{\sigma, \bullet}$ by $\gr_\sigma(A)$ or $A_\sigma$. Clearly, if the face $\sigma$ is the whole valuation cone $\V_X$ then the associated graded $\gr_\sigma(A)$ is just the horospherical degeneration $A_{\hc}$.

\begin{Prop} \label{prop-partial-horo-degeneration-gr_v}
Let $v \in \V_X$ be a $G$-invariant valuation. Suppose $v$ lies in the relative interior of a face $\sigma$ of $\V$. Then the graded algebra 
$\gr_v(A)$ is naturally isomorphic to the partial horospherical contraction $\gr_\sigma(A)$.
\end{Prop}

\begin{Rem} \label{rem-geo-realize-partial-horo}
There is a geometric interpretation of $\Spec(A_\sigma)$ as follows. Let $v$ be in the relative interior of $\sigma$ and let $X_v$ be the elementary spherical embedding corresponding to the fan consisting of the single ray generated by $v$. Let $D_v$ denote the 
divisor at infinity in $X_v$. It is the unique closed $G$-orbit in $X_v$. Then the deformation $\Spec(A_\sigma)$ contains the normal bundle in $X_v$ of the divisor $D_v$ as an open subset. Following V. Batyrev we refer to the stabilizer of a general point in this normal bundle as a {\it satellite subgroup} associated to the face $\sigma$ (\cite{Batyrev-Moreau}).
Moreover, one can glue together the varieties $\Spec(A_\sigma)$ in a family. Let $\mathcal{T}_X$ denote the tail cone of $X$ (which is the dual cone to $-\V_X$). Then one can define a family $\pi: \mathfrak{X} \to \Spec(\k[\mathcal{T}_X \cap \Lambda_X])$ 
such that the fibers are the partial horospherical contractions $\Spec(A_\sigma)$ (see \cite{AB-moduli}).
\end{Rem}

\begin{Rem} 
As far as we know, the material in this section (namely Definition \ref{def-partial-order-corr-face} and Proposition \ref{prop-partial-horo-degeneration-gr_v}) are not quite new and have been observed by other authors (see for example \cite{Avdeev-Cupit-Foutou}). 
\end{Rem}

\begin{proof}[Proof of Proposition \ref{prop-partial-horo-degeneration-gr_v}] 
For $\lambda, \mu \in \Lambda_X^+$ let $m_X$ denote the multiplication map
\begin{equation} \label{equ-m_X}
m_X: W_\lambda \otimes W_\mu \to \bigoplus_{\lambda + \mu - \eta \textup{ tail}} W_\eta.
\end{equation}
Recall that $\eta$ is a tail if $\eta \geq_X \lambda + \mu$. The map $m_X$ then gives a map $\gr_\sigma m_X$:
\begin{equation} \label{equ-gr-sigma-m_X}
\gr_\sigma m_X: W_\lambda \otimes W_\mu \to (\bigoplus_{\lambda + \mu - \eta \textup{ tail}} W_\eta) / (\bigoplus_{\eta' >_\sigma \lambda + \mu} W_{\eta'}),
\end{equation}
which defines the multiplication in the algebra $A_\sigma$.
Similarly, for a valuation $v \in \V_X$ we get a map $\gr_v m_X$:
\begin{equation} \label{equ-gr-v-m_X}
\gr_v m_X: W_\lambda \otimes W_\mu \to 
(\bigoplus_{\lambda + \mu - \eta \textup{ tail}} W_\eta) / (\bigoplus_{\langle v, \eta' \rangle > \langle v, \lambda + \mu \rangle} 
W_{\eta'}),
\end{equation}
which defines the multiplication in the algebra $\gr_v(A)$. We would like to show that these two multiplications coincide. As usual 
let $\{ \beta_1, \ldots, \beta_\ell\}$ denote the set of spherical roots. Let $S \subset \{ \beta_1, \ldots, \beta_\ell\}$ be the subset of spherical roots determining the 
face $\sigma$ as in \eqref{equ-face-sigma}.
Then the relative interior $\sigma^\circ$ of $\sigma$ is defined by the inequalities:
$$\sigma^\circ = \{ v \in \V_X \mid \langle v, \beta \rangle = 0, ~\forall \beta \in S \textup{ and } ~ \langle v, \beta' \rangle > 0, ~\forall \beta' \notin S \}.$$
Let $\eta$ be a weight appearing in the righthand side of \eqref{equ-m_X} which means $\lambda + \mu - \eta$ is a tail. 
Also take a valuation $v$ in the relative interior $\sigma^\circ$. Since $\lambda + \mu - \eta$ is a tail we can write $\lambda + \mu - \eta = \sum_i c_i \beta_i$ where 
$c_i \geq 0$ for all $i$. Then $\eta$ appears in the denominator in the righthand side of \eqref{equ-gr-v-m_X} if and only if: 
$$\langle v, \sum_i c_i \beta_i \rangle  = \langle v, \sum_{\beta_i \notin S} c_i \beta_i \rangle  = 
\sum_{\beta_i \notin S} c_i \langle v, \beta_i \rangle > 0.$$ Since 
$v$ is in the relative interior of $\sigma$ this is the case if and only if there is $\beta_i \notin S$ such that $c_i > 0$. 
That is, if and only if $\eta >_\sigma \lambda + \mu$. This finishes the proof.
\end{proof}

\subsection{Spherical Gr\"obner fan} \label{sec-spherical-Grobner-fan}
In this section we introduce a generalization of the notion of Gr\"obner fan of an ideal in a polynomial ring. As in the usual Gr\"obner theory it is more natural to work with homogeneous ideals. Thus, in this section we assume that $A$ is a $\z_{\geq 0}$-graded $G$-algebra and domain and the action of $G$ respects the grading. Moreover, each graded piece $A_i$ is a multiplicity free $G$-module. Thus $A$ is a multiplicity free $(\k^* \times G)$-algebra. We let $\V_A$ denote the $(\k^* \times G)$-invariant valuation cone of $A$. Similarly, the weight lattice $\Lambda_A$ and the weight semigroup $\Lambda_A^+$ are for the $(\k^* \times G)$-action on $A$.

We begin with introducing the notion of an initial ideal with respect to an invariant valuation $v$. In fact, for the next couple of definitions we do not need to assume that $A$ is graded and the definitions make sense in the non-graded case as well.

\begin{Def}[Initial ideal with respect to a valuation] \label{def-initial-ideal-v}
Let $I \subset A$ be an ideal. We denote by $\In_v(I)$ the ideal in $\gr_v(A)$ generated by 
all the $\In_v(f)$ where $f \in I$. 
By Proposition \ref{prop-partial-horo-degeneration-gr_v} we may consider $\In_v(I)$ as an ideal in $A_\sigma$ where $\sigma$ is the unique face of the valuation cone $\V_A$
such that $v$ lies in the relative interior of $\sigma$.
\end{Def}

{Let $v \in \V_A$ be an invariant valuation. Also let $\succ$ be a total order on the weight lattice $\Lambda_A$ as in Section \ref{subsec-sph-Grobner}. 
We define the total order $\succ_v$ as follows: $\lambda \succ_v \mu$ if either 
$\langle v, \lambda \rangle > \langle v, \mu \rangle$, or $\langle v, \lambda \rangle = \langle v, \mu \rangle$ and $\lambda \succ \mu$.} 

We note that if $v$ attains positive values on $A$ then the total order $\succ_v$ is not maximum well-ordered. By Proposition \ref{prop-inv-val-min-isotypic}
if $v$ has a positive value on $A$ then we can find $\lambda \in \Lambda_A^+$ and $f_\lambda \in W_\lambda$ such that $v(f_\lambda) > 0$, or in other words
$\langle v, \lambda \rangle > 0$. By the definition of $\succ_v$ we then have $\lambda \prec_v 2\lambda \prec_v 3\lambda \prec_v $. This shows that 
$\Lambda_A^+$ is not maximum well-ordered with respect to $\succ_v$.
One can verify the following.
\begin{Prop} \label{prop-nonneg-valuation}
With notation as above, the total order $\succ_v$ refines the spherical dominant order. 
Moreover, if $v$ is nonpositive on $A$ then $\Lambda_A^+$ is maximum well-ordered with respect to $\succ$. Thus $\succ_v$ satisfies the properties in
Assumption \ref{ass-total-order}. 
\end{Prop}

Thus whenever we deal with a total order of the form $\succ_v$ we would like $v$ to be nonpositive on $A$.

\begin{Def} \label{def-equiv-valuation}
Let $I \subset A$ be an ideal. Given two invariant valuations 
$v, w \in \V_A$, we say $v \sim w$ if:
\begin{itemize}
\item[(i)] $v$ and $w$ lie in the relative interior of the same face $\sigma$ of the valuation cone $\V_A$.
\item[(ii)] $\In_v(I) = \In_w(I)$ regarded as subsets of $A_\sigma$ (under the isomorphisms $\gr_v(A) \cong A_\sigma \cong \gr_w(A)$ coming from
Proposition \ref{prop-partial-horo-degeneration-gr_v}). 
\end{itemize}
\end{Def}


The following lemma will be used below. We skip the proof. It is a straightforward from the definitions.
\begin{Lem} \label{lem-in_v-in_succ}
Let $v \in \V_A$ be a $G$-invariant valuation and $\succ$ a total order on $\Lambda_A$ as above. For any ideal $I$ we have:
$$\In_{\succ_v}(I) = \In_\succ(\In_v(I)),$$
where both sides are considered as ideals in the horospherical contraction $A_{\hc}$ of $A$.
\end{Lem}


\begin{Rem}   \label{rem-graded-alg-Grobner-fan}
In general, the (closures of) equivalence classes of the equivalence relation $\sim$ on $\V_A$ may not be convex. But as we will see below, this is the case when the ideal $I$ is homogeneous with respect to the $\z_{\geq 0}$-grading of $A$.
\end{Rem}

From here on the assumption that $A$ is graded becomes important. 
Let $\deg: A \setminus \{0\} \to \z_{\geq 0}$ be the degree function of the grading of $A$. We note that the map $f \mapsto -\deg(f)$ is a $(\k^* \times G)$-invariant valuation on $A$. The following lemma is important.
\begin{Lem} \label{lem-nonpos-val-graded-alg}
Let $v \in \V_A$ be an invariant valuation. We have the following.
\begin{itemize}
\item[(a)] For sufficiently large $k$ the valuation $v' = v - k\deg$ is a nonpositive invariant valuation on $A$. (Here $v - k\deg$ is regarded as an element in the invariant valuation cone $\V_A$.) 
\item[(b)] The valuation $v'$ lies on the relative interior of the same face as $v$ and hence $\gr_{v'}(A)$ is naturally isomorphic to $\gr_v(A)$.
\item[(c)] Let $I \subset A$ be a homogeneous ideal. Then under the isomorphism in (b) we have $\In_{v'}(I) \cong \In_v(I)$.  
\end{itemize}
\end{Lem}
\begin{proof}
{One verifies that the valuation $-\deg$ lies on a generating ray of the simplicial cone $\V_A$.} 
Thus we see that 
$v$ and $v'$ lie on the relative interior of the same face of the cone. Let $h_1, \ldots, h_\ell$ be $(B \times \k^*)$-eigenvectors in $A$ 
whose weights generates the lattice $\Lambda_A$. By Proposition 
\ref{prop-inv-val-min-isotypic} 
every invariant valuation is uniquely determined by its values on the $h_i$. Now we can choose $k$ sufficiently large so that $v'(h_i) \leq 0$ for all $i=1, \ldots, \ell$. It follows that $v'$ is 
nonpositive on the whole $A$. It remains to show that $I$ has the same initial ideals with respect to $v$ and $v'$. To do this we notice that since $I$ is a $\z$-homogeneous ideal we can find a 
universal spherical Gr\"obner basis $\G = \{f_1, \ldots, f_s\}$ for $I$ consisting of $\z$-homogeneous elements, i.e. $f_i \in A_{d_i}$ for some $d_i \geq 0$. It follows from the 
construction of $v'$ that $\In_{v}(f_i) = \In_{v'}(f_i)$, for all $i$. On the other hand, since $\G$ is a spherical Gr\"obner basis we know that $\In_{v}(\G)$ and $\In_{v'}(\G)$ generate the initial ideals
$\In_v(I)$ and $\In_{v'}(I)$ respectively. Thus $\In_v(I) = \In_{v'}(I)$ as required. 
\end{proof}


We can now define the notion of spherical Gr\"obner fan of a homogeneous ideal $I \subset A$.
\begin{Def}[Spherical Gr\"obner fan] \label{def-Grobner-fan}
We call the set of closures of equivalence classes of $\sim$ in $\V_A$,
the {\it spherical Gr\"obner fan of $I$} and denote it by $\textup{GF}(I)$.   
\end{Def}
Below we will show that $\textup{GF}(I)$ is indeed a fan.

\begin{Th}
The equivalence classes of $\sim$ are relatively open rational polyhedral convex cones. 
\end{Th}
\begin{proof}
The proof follows the usual Gr\"obner theory proof from \cite[Proposition 2.3]{Sturmfels}.
Let $\sigma$ be a face of the simplicial cone $\V_A$ and let $C[v]$ denote the equivalence class of a valuation 
$v \in \V_A$ which lies in the relative interior of $\sigma$. Fix a total order $\succ$. By Lemma \ref{lem-nonpos-val-graded-alg} we can assume 
that $v$ is nonpositive on $A$ and hence the associated total order $\succ_v$ has the good properties in Assumption \ref{ass-total-order}. 
Take a reduced spherical Gr\"obner basis $\G$ for $I$ with respect to $\succ_v$ 
(Definition \ref{def-reduced-sph-Grobner-basis}). We will prove the following:
\begin{equation} \label{equ-C[v]}
C[v] = \{ v' \in \sigma \mid \In_v(f) = \In_{v'}(f), ~\forall f \in \G \}.
\end{equation}
We note that the righthand side of \eqref{equ-C[v]} defines a relatively open polyhedral convex cone 
and hence it proves the claim of the theorem.
First we prove the ``$\supset$'' part. Let $v'$ lie in the righthand side of \eqref{equ-C[v]}. We observe that the image of $\G$ in $\gr_v(A)$ is a spherical Gr\"obner basis (with respect to $\succ_v$) 
for the ideal $\In_v(I)$. Thus $$\In_v(I) = \langle \In_v(f) \mid f \in \G \rangle = \langle \In_{v'}(f) \mid f \in \G \rangle \subset \In_{v'}(I).$$
It follows that $\In_{\succ_v}(I) \subset \In_{\succ_{v'}}(I)$ as ideals in $A_{\hc}$. {Since both quotients $A_{\hc} / \In_{\succ_v}(I)$ and $A_{\hc}/ \In_{\succ_{v'}}(I)$ are isomorphic to $A/I$, 
this containment cannot be strict and hence 
$\In_{\succ_v}(I) = \In_{\succ_{v'}}(I)$.} From Proposition \ref{prop-sph-Grobner-basis-generate-ideal} applied to the ideals $\In_v(I) \subset \gr_v(A) = A_\sigma$ and 
$\In_{v'}(I) \subset \gr_{v'}(A) = A_\sigma$, we then conclude that $\In_v(I) = \In_{v'}(I)$. Next we prove the inclusion ``$\subset$''. Take $v' \in C[v]$. Then $\In_v(I) = \In_{v'}(I)$.
We would like to show that for any $f \in \G$ we have $\In_v(f) = \In_{v'}(f)$. As said above the image $\In_v(\G)$ of $\G$ in $\gr_v(A)$ is a spherical Gr\"obner basis (with respect to $\succ_v$) for 
$\In_v(I) = \In_{v'}(I)$. Take $f \in \G$. Thus $\In_{v'}(f) \in \In_v(I)$ can be reduced to $0$ using $\In_v(\G)$. Since $\G$ was a reduced spherical Gr\"obner basis, the isotypic component of 
$f$ that appears in $\In_{v'}(f)$ is $\In_{\succ_v}(f)$. Let us write $\In_v(f) = \In_{\succ_v}(f) + h$ and $\In_{v'}(f) = \In_{\succ_v}(f) + h'$ 
where none of the isotypic components of $h$ and $h'$ ly in $\In_{\succ_v}(I)$. On the other hand, after the first step of reducing $\In_v(f)$ to $0$ using $\In_v(\G)$ we obtain $h' - h$ which lies in 
$\In_v(I)$. This shows that $h' - h$ must be equal to $0$ and hence $\In_v(f) = \In_{v'}(f)$ as required.
\end{proof}

Finally we show that the spherical Gr\"obner fan is indeed a fan. We first need the following analogue of the notion of a Newton polytope.
\begin{Def}[Generalized Newton polytope] \label{def-general-Newton-polytope}
Let $f \in A$ and write $f = \sum_\lambda f_\lambda$ as a sum of its isotypic components. We define the convex polytope $\Delta(f) \subset \V_A$ to be the convex hull of the 
support of $f$, that is:
$$\Delta(f) = \conv\{ \lambda \mid f_\lambda \neq 0\}.$$
\end{Def}

\begin{Rem} \label{rem-KKh-horosph}
The polytope $\Delta(f)$ appears in \cite{KKh-horosph} where it is used to prove a version of the Bernstein-Kushnirenko theorem from toric geometry for horospherical varieties.
\end{Rem}

\begin{Th} \label{th-sph-Grobner-fan}
The spherical Gr\"obner fan is a fan.
\end{Th}
\begin{proof}
The proof is along the same line as the usual Gr\"obner theory proof from \cite[Proposition 2.4]{Sturmfels}. Take a valuation $v \in \V_A$ and let $\overline{C[v]}$ denote the closure of the 
open cone $C[v]$ in $\V_A$. Let $\succ'$ be a total order and let $\succ'_v$ be the total order associated to $\succ'$ and $v$. To simplify the notation we denote $\succ'_v$ by $\succ$. Note that by definition $\succ_v$ coincides with $\succ$. Let $\G$ be a reduced spherical Gr\"obner basis for the ideal $I$ with respect to $\succ$.
Consider the polytope $$\Delta = \sum_{f \in \G} \Delta(f).$$
Any valuation $v \in \V_X$ defines a face $\textup{face}_v(\Delta)$ on which the linear function $\langle v, \cdot \rangle$ attains its minimum.
We note that from the equation \eqref{equ-C[v]} it follows that the closure $\overline{C[v]}$ is the dual cone to the face $\textup{face}_v(\Delta)$ of the polytope 
$\Delta$ corresponding to $v$. Next take $v' \in \overline{C[v]}$. One shows that for any $f \in A$, $\In_{\succ_{v'}}(f) = \In_\succ(\In_{v'}(f)) = \In_\succ(f)$. Thus, $\G$ is a reduced spherical Gr\"obner basis for the total order $\succ_{v'}$ as well, and hence we see that $C[v]$ and $C[v']$ are the dual cones to the faces of the same polytope $\Delta$ corresponding to the valuations $v$ and $v'$ respectively.
But $\textup{face}_v(\Delta)$ is a face of $\textup{face}_{v'}(\Delta)$ and consequently the closed convex cone $\overline{C[v']}$ is a face of the closed convex cone $\overline{C[v]}$.
Using this one can show that the collection of the closures $\overline{C[v]}$ of the equivalence classes of $\sim$ satisfies the defining axioms of a fan.
\end{proof}


Recall that we call an ideal $J \subset A_{\hc}$ a $\Lambda_X$-homogeneous ideal if it is generated by a finite number of $\Lambda_X$-homogeneous elements 
(Definition \ref{def-Lambda-A-homog-ideal}).
The next proposition is a generalization of the fact that for an ideal in a polynomial ring the initial ideal corresponding to a generic choice of a weight is monomial.
It is a corollary of the existence of a universal spherical Gr\"obner basis. The proof is verbatim to the case of a polynomial ring but we include it here for the sake of completeness.
\begin{Prop} \label{prop-generic-initial-ideal-is-monomial}
Let $I \subset A$ be an ideal and let $\sigma$ be a cone of {maximum dimension} in the Gr\"obner fan $\textup{GF}(I)$ (i.e. $\sigma$ has the same dimension as the valuation cone 
$\V_A$). Then for any valuation $v$ in the interior of $\sigma$ the initial ideal $\In_v(I)$ is a $\Lambda_X$-homogeneous ideal in $\gr_v(A) \cong A_{\hc}$, the horospherical contraction of $A$. 
\end{Prop}
\begin{proof}
Fix a universal spherical Gr\"obner basis $\G=\{f_1, \ldots, f_s \}$ for $I$. Also let $\succ$ be an ordering on $\Lambda_X$. 
Let $v$ be in the interior of $\sigma$. Take a vector $w$ in the linear span of $\V_A$ and consider 
$v' = v + \epsilon w$ where $\epsilon > 0$. If $\epsilon$ is sufficiently small then $v'$ also lies in the interior of $\sigma$ and hence $\In_{v'}(I) = \In_{v}(I)$. Thus it suffices to prove 
that $\In_{v'}(I)$ is $\Lambda_X$-homogeneous. We show that this is the case if $w$ is generic enough. In fact one verifies that if $w$ is generic then all the initial elements 
$\In_{v'}(f_i)$ are $\Lambda_X$-homogeneous. Now since $\G$ is universal it is a spherical Gr\"obner basis for $I$ with respect to $\succ_{v'}$. It follows that 
$\In_{v'}(\G)$ is a spherical Gr\"obner basis for $\In_{v'}(I)$ with respect to $\succ_{v'}$ as well. Thus $\In_{v'}(I)$ is generated by the 
$\Lambda_X$-homogeneous elements $\In_{v'}(f_i)$ and hence is a $\Lambda_X$-homogeneous ideal as desired.
\end{proof}

\section{Spherical tropical varieties via ideals} \label{sec-sph-trop-var-ideals}
Generalizing the notion of tropical variety of a subvariety of the torus $(\k^*)^n$, in this section we give our first construction of
spherical tropical varieties, namely construction using initial ideals with respect to valuations (Definitions \ref{def-sph-trop-var-X_B} and \ref{def-sph-trop-var-G/H}). It turns out that in the context of spherical varieties, it is more natural to consider subvarieties in the open Borel orbit (in fact $G/H$ may not be affine or even quasi-affine while the open Borel orbit is always affine). In Section \ref{sec-sph-trop-var-trop-map}, following Vogiannou, we define the spherical tropical variety using the tropicalization map (Definition \ref{def-sph-tropicalization}). The content of our fundamental theorem (Theorem \ref{th-fundamental}) is that these two constructions coincide.

\subsection{Spherical tropical variety of a subscheme in the open $B$-orbit} \label{subsec-sph-trop-var-X_B}
In this section we define the notion of a spherical tropical variety for a subscheme in the open Borel orbit. 

Let $v \in \V_{G/H}$ be a $G$-invariant valuation. Recall that $X_v$ denotes the equivariant embedding of 
${G/H}$ corresponding to the ray generated by $v$. It consists of two $G$-orbits: the open $G$-orbit ${G/H}$ and the $G$-stable prime divisor $D_v$. 
Fix a Borel subgroup $B$ and let $X_B$ be the open $B$-orbit in $G/H$. One knows that $X_B$ is an affine variety.
We denote the set of $B$-stable prime divisors in $G/H$ by $\mathcal{D}(G/H)$. One observes that:
$$X_B = G/H ~\setminus \bigcup_{D \in \mathcal{D}(G/H)} D.$$ 
We also denote the open $B$-orbit in the $G$-orbit $D_v$ by $D'_v$.

Similarly let us denote the set of $B$-stable prime divisors in $X_v$ by $\mathcal{D}(X_v)$.
One defines a subvariety $X_{v, B} \subset X_v$ by:
$$X_{v, B} = X_v \setminus \bigcup_{D \in \mathcal{D}(X_v) \setminus \{D_v\}} D.$$ 

One shows the following (see \cite[Theorem 2.1]{Knop-LV}):
\begin{Th} \label{th-X_v_B}
$X_{v, B}$ is a $B$-stable affine subvariety of $X_v$ and $X_{v, B} \cap D_v$ is the $B$-orbit $D_v'$. Moreover the coordinate ring of 
$X_{v, B}$ can be described as:
$$\k[X_{v, B}] = \{ f \in \k[X_B] \mid v(f) \geq 0 \}.$$
\end{Th}

\begin{Rem} \label{rem-Knop-X_0-X_1}
More generally, for each closed $G$-orbit in a spherical variety one can define a $B$-stable affine neighborhood. See \cite[Section 2]{Knop-LV}
for more details (in \cite{Knop-LV} this affine neighborhood is denoted by $X_0$).  
\end{Rem}

Now we define analogues of the notions from Sections \ref{subsec-sph-Grobner}, \ref{subsec-partial-horo} and \ref{subsec-sph-trop-var-G/H} replacing the algebra $A$ by $\k[X_B]$.
Namely, let $\gr_v(\k[X_B])$ denote the associated graded of the filtration on $\k[X_B]$ defined by the valuation $v$. More precisely,
for every $a \in \q$ we put $\k[X_B]_{v \geq a} = \{ f \in \k[X_B] \mid v(f) \geq a \}.$
Similarly, one defines the subspace $\k[X_B]_{v > a}$. Then $$\gr_v(\k[X_B]) = \bigoplus_{a \in \q_{\geq 0}} \k[X_B]_{v \geq a} / \k[X_B]_{v > a}.$$
We note that each subspace in the filtration is $B$-stable and the algebra $\gr_v(\k[X_B])$ is naturally a $B$-algebra.

\begin{Rem} \label{rem-gr_v-X_B}
{Using Proposition \ref{prop-partial-horo-degeneration-gr_v}, in Section \ref{subsec-fan-structure} we will see that if $v_1, v_2$ lie in the relative interior of the same face $\sigma$ of the valuation 
cone $\V_{G/H}$ then the corresponding associated graded algebras $\gr_{v_1}(\k[X_B])$ and $\gr_{v_2}(\k[X_B])$ are naturally isomorphic.}
\end{Rem}

For each $f$ with $v(f) = a$ let $\In_v(f)$ denote the image of $f$ in the quotient space 
$\k[X_B]_{v \geq a} / \k[X_B]_{v > a}$. 
Finally for an ideal $J \subset \k[X_B]$ we let $\In_v(J)$ be the ideal in $\gr_v(\k[X_B])$ generated by all the $\In_v(f)$, $f \in J$.
We can now define the notion of spherical tropical variety of an ideal in the coordinate ring $\k[X_B]$. 
\begin{Def}[Spherical tropical variety of an ideal in the coordinate ring of $X_B$]    \label{def-sph-trop-var-X_B}
Let $J \subset \k[X_B]$ be an ideal.
We define $\trop_{B}(J)$ to be: 
$$\trop_B(J) = \{v \in \V_{G/H} \mid \In_v(J) \neq \gr_v(\k[X_B])\}.$$ 
In other words, $\trop_B(J)$ consists of $v \in \V_{G/H}$ such that $\In_v(J)$ does not contain a unit element. We call $\trop_{B}(J)$ the {\it spherical tropical variety} of $J$. When $J = \langle h \rangle$ is a principal ideal, we 
write $\trop_B(h)$ in place of $\trop_B(J)$ and call it the {\it spherical tropical hypersurface} of $h$. 
\end{Def}
In Section \ref{subsec-GxG} we consider two examples of tropical hypersurfaces in an open Borel orbit in $\GL(2, \k)$.

\begin{Rem}  \label{rem-trop-B-supp-fan}
In Section \ref{subsec-fan-structure} we use spherical Gr\"obner theory to show that the spherical tropical variety $\trop_B(J)$ is indeed the support of a rational polyhedral fan (Theorem \ref{th-fan-st-sph-trop-var}). 
\end{Rem}

The following result justifies the above definition.
\begin{Th} \label{th-sph-trop-var-X_v_B}
Let $Z \subset X_B$ be the subscheme of the open $B$-orbit defined by an ideal $J \subset \k[X_B]$.
Let $v \in \V_{G/H}$ be a valuation. Then $v$ lies in $\trop_B(J)$ if and only if the closure $\overline{Z}$ of $Z$ in $X_{v, B}$ intersects 
the divisor at infinity $D_v'$.  
\end{Th}

\begin{proof}
Consider the Rees algebra: 
$$\mathcal{R}_v(\k[X_B]) = \bigoplus_{a \geq 0} \k[X_B]_{v \geq a},$$
where the direct sum is over all $a \in \q_{\geq 0}$ which lie in the value semigroup $v(\k[X_B] \setminus \{0\})$.
In light of Theorem \ref{th-X_v_B}, we see that $\Proj(\mathcal{R}_v(\k[X_B]))$ is the blowup of the variety $X_{v, B}$ 
along the prime divisor $D_v'$. Also the exceptional divisor in the blowup is $\Proj(\gr_v(\k[X_B])$.
But the blowup of a variety along a prime Cartier divisor coincides with the variety itself, so we conclude that $X_{v, B} = \Proj(\mathcal{R}_v(\k[X_B]))$ and
$D'_v = \Proj(\gr_v(\k[X_B])$.   

The filtration associated to the valuation $v$ on $\k[X_B]$ induces a  pushforward filtration $F_{\bullet}$ on the quotient algebra 
$\k[X_B]/J$. For any $a \in \q_{\geq 0}$ we simply define the subspace $F_{\geq a}$ to be the image of $\k[X_B]_{v \geq a}$ in $\k[X_B]/J$ 
under the natural homomorphism $\k[X_B] \to \k[X_B]/J$. The subspace $F_{> a}$ is defined similarly. The next lemma relates this pushforward filtration with the notion of initial ideal. Its proof is straightforward (see also \cite[Lemma 4.4]{KM-Khov-bases}).

\begin{Lem}[Initial ideal in terms of pushforward filtration] \label{lem-pushforward-initial-ideal}
There is a natural isomorphism between the associated graded $\gr_{F_\bullet}(\k[X_B]/ J)$ 
of the pushforward filtration $F_\bullet$ and the quotient algebra $\gr_v(\k[X_B]) / \In_v(J)$.
\end{Lem}

Next we notice that the scheme-theoretic intersection $\overline{Z} \cap D_v'$ can be constructed as $\Proj(\gr_{F_\bullet}(\k[X_B]/ J))$ which by 
Lemma \ref{lem-pushforward-initial-ideal} is equal to $\Proj(\gr_v(\k[X_B]) / \In_v(J))$. Since $\Proj$ of a positively graded algebra is nonempty if and only if the algebra is nonzero, it follows that $\overline{Z} \cap D_v' \neq \emptyset$ if and only if $\In_v(J) \neq \k[X_B]$. This finishes the proof.
\end{proof}

\subsection{Fan structure}  \label{subsec-fan-structure}
In this section we use the existence of the spherical Gr\"obner fan (Theorem \ref{th-sph-Grobner-fan}) to show that the spherical tropical variety of an ideal in the open Borel orbit is the support of a rational polyhedral fan. This is a generalization of the situation in the classical torus case (see for example \cite[Chapter 2]{Maclagan-Sturmfels}).

We begin by setting the stage. We follow the notation from the previous sections. Let $\overline{X}$ be a projective spherical embedding of a spherical homogeneous space $G/H$. Let $L$ be a very ample $G$-line bundle on $\overline{X}$. 
Without loss of generality we assume there is a $B$-eigensection $s \in H^0(\overline{X}, L)$ that vanishes on all the $B$-stable divisors in $\overline{X}$. 

Let $A = \bigoplus_{i \geq 0} A_i$, $A_i = H^0(\overline{X}, L^{\otimes i})$, be the algebra of sections of $L$. It is a finitely generated multiplicity-free $(\k^* \times G)$-algebra and $\dim(A_i) < \infty$, for all $i$. We will apply the spherical Gr\"obner theory to this algebra (by Remark \ref{rem-good-ordering-exists} this algebra admits an ordering satisfying Assumption \ref{ass-total-order}). We denote the semigroup of highest weights of the algebra $A$ by $\tilde{\Lambda}^+_A \subset \z_{\geq 0} \times \Lambda$. The lattice generated by $\tilde{\Lambda}^+_A$ is $\tilde{\Lambda}_A$. We also denote the cone of $(\k^* \times G)$-invariant valuations on $A$ by $\tilde{\V}_A$ (we use the tilde notation to distinguish between the corresponding notions for the $(\k^* \times G)$-action as opposed to the $G$-action).

Given the section $s$, define an algebra homomorphism $\pi: A \to \k[X_B]$ as follows. For any $i \geq 0$ and $f_i \in A_i$ let $\pi(f_i) = f_i / s^i$. Since $s$ is a $B$-eigensection it does not have any zeros in $X_B$ and hence the image of $\pi$ lies in $\k[X_B]$. Moreover, because $L$ is very ample, we have an embedding $X_B \subset G/H \subset \overline{X} \hookrightarrow \p(H^0(\overline{X}, L)^*)$. Since $s$ vanishes on all the $B$-stable divisors, $X_B$ embeds as a closed subvariety in the affine space $\p(H^0(\overline{X}, L)^*) \setminus \{s=0\}$. It thus follows that $\pi$ is surjective. 

Next, we define an embedding $\phi$ from the valuation cone $\V_{G/H}$ into the valuation cone $\tilde{\V}_A$. Take $v \in \V_{G/H}$ and define the valuation $\tilde{v} = \phi(v)$ on $A$ as follows. For $f = \sum_i f_i$, $f_i \in A_i$, put:
\begin{equation} \label{equ-tilde-v}
\tilde{v}(f) = \min\{ v(f_i/s^i) \mid f_i \neq 0\}.
\end{equation}
\begin{Prop} \label{prop-tilde-v-valuation}
(1) The function $\tilde{v}: A \setminus \{0\} \to \q$ is a $(\k^* \times G)$-invariant valuation. (2) The map $\phi: v \mapsto \tilde{v}$ is a one-to-one linear map from the cone $\V_{G/H}$ to the cone $\tilde{\V}_A$ and the image $\phi(\V_{G/H})$ is the intersection of $\tilde{\V}_A$ with a hyperplane. 
\end{Prop}
\begin{proof}
(1) One checks that \eqref{equ-tilde-v} defines a valuation. Take $f = \sum_i f_i \in A$. By the definition of the action of $\k^*$ on $A$, for $t \in \k^*$ we have 
$t \cdot f = \sum_i t^if_i$, and thus $\tilde{v}(t \cdot f) = \min\{ v(t^i f_i / s^i) \mid f_i \neq 0\} = \min\{ v(f_i / s^i) \mid f_i \neq 0\} = \tilde{v}(f)$. Thus $\tilde{v}$ is $\k^*$-invariant. 
It remains to show that $\tilde{v}$ is $G$-invariant. Take $g \in G$, it is enough to show that for any $k>0$ and $f_k \in A_k$ we have $\tilde{v}(g \cdot f_k) = \tilde{v}(f_k)$. 
One knows that the valuation $v$ on $\k(G/H)$ lifts to a $G$-invariant valuation $\hat{v}$ on $\k(G)$ for the left $G$-action (see \cite[Corollary 1.5]{Knop-LV}). Also one knows that $H^0(G/H, L_{|G/H})$ can be identified with the space of $H$-eigensection in $\k[G]$ corresponding to an $H$-weight $\chi$. With this identification we have  $v((g \cdot f_k) / s^k) = \hat{v}(g \cdot f_k) - \hat{v}(s^k) = \hat{v}(f_k) - \hat{v}(s^k) = \hat{v}(f_k / s^k) = v(f_k/s^k)$. (2) The injectivity and linearity of the map $\phi$ are straightforward from the definition. It remains to prove the last assertion. The valuation cone $\tilde{\V}_A$ sits in the rational vector space $\tilde{\mathcal{Q}}_A = \Hom(\tilde{\Lambda}_A, \q)$ (Section \ref{subsec-inv-valuation}). We note that the set $\{\tilde{v} \mid \tilde{v}(s) = 0 \}$ is a hyperplane in this vector space. Now let $\tilde{v} \in \tilde{\V}_A$ be a valuation with $\tilde{v}(s) = 0$. Define a function $v$ on $\k[X_B]$ as follows. For $h \in \k[X_B]$, write $h = f_i / s^i$, for some $i \geq 0$ and $f_i \in A_i$. Let $v(h) = \tilde{v}(f_i)$. One verifies that the function $v$ is well-defined and gives a $G$-invariant valuation on $G/H$. This finishes the proof.   
\end{proof}

\begin{Rem}  \label{rem-torus-case}
The above is a generalization of the classical torus case where $G = B = T = (\k^*)^n$ and $H=\{e\}$. In this case, we can take $\overline{X}$ to be the projective space $\p^n$, $L = \mathcal{O}(n+1)$ and $s = x_1 \cdots x_{n+1}$, where $(x_1, \ldots, x_n)$ are coordinates on $T$ and $(x_1 : \cdots : x_{n+1})$ are homogeneous coordinates on $\p^n$.
\end{Rem}

\begin{Th}[Fan structure on a spherical tropical variety]  \label{th-fan-st-sph-trop-var}
With notation as before, let $J \subset \k[X_B]$ be an ideal. Then the spherical tropical variety $\trop_B(J)$ is the support of a rational polyhedral fan. Moreover, this fan structure comes from intersecting the image of $\trop_B(J)$ with a spherical Gr\"obner fan for a homogeneous ideal $\tilde{J}$ in the algebra of sections $A$ (of a very ample line bundle on a projective spherical embedding of $G/H$).
\end{Th}
\begin{proof}
Let $\tilde{J} \subset A$ be the ideal generated by the homogeneous elements in $\pi^{-1}(J)$. By definition $\tilde{J}$ is a homogeneous ideal. It is a generalization of the of {\it homogenization} of an ideal in the Laurent polynomial algebra $\k[x_1^{\pm 1}, \ldots, x_n^{\pm 1}]$ 
(see \cite[Section 2.6]{Maclagan-Sturmfels}). 
Consider the image of $\trop_B(J)$ in the valuation cone $\tilde{\V}_A$ under the linear embedding $\phi$ (Proposition \ref{prop-tilde-v-valuation}). We show that this image is a subfan of the spherical Gr\"obner fan $\textup{GF}(\tilde{J})$, i.e. it is a union of closures of cones from $\textup{GF}(\tilde{J})$. Let $v \in \trop_B(J)$, i.e. $\gr_v(\k[X_B])$ does not contain any unit element. Let $\overline{C[\tilde{v}]}$ be the closure of cone in the Gr\"obner fan containing $\tilde{v} \in \tilde{\V}_A$. 
Suppose $v' \in \V_{G/H}$ is such that $\tilde{v}' \in \overline{C[\tilde{v}]}$. This implies that there are faces $\sigma$, $\sigma'$ of the cone $\tilde{\V}_A$ such that $\sigma' \subset \overline{\sigma}$ and $\tilde{v}$, $\tilde{v}'$ belong to the relative interiors of $\sigma$, $\sigma'$ respectively. One verifies that the map $\In_{\tilde{v}'}(f) \mapsto \In_{\tilde{v}}(\In_{\tilde{v}'}(f))$, $\forall f \in A$, defines a multiplicative homomorphism from $\In_{\tilde{v}'}(\tilde{J})_\hom$ to $\In_{\tilde{v}}(\tilde{J})_\hom$ (here the subscript $\hom$ denotes the set of homogeneous elements, with respect to the $\q$-grading in the associated graded, in the corresponding set). 
It then follows that we have a multiplicative homomorphism from $\In_{v'}(J)$ to $\In_v(J)$. If $\In_{v'}(J)$ contains a unit element then $\In_v(J)$ should also contain a unit element which contradicts the assumption $v \in \trop_B(J)$. Thus $v' \in \trop_B(J)$. This finishes the proof. 
\end{proof}

\subsection{Spherical tropical hypersurfaces}  \label{subsec-sph-trop-hypersurface}
In this section we see how to compute the spherical tropical variety of a hypersurface in the open Borel orbit.
We follow the notation from previous sections. Let $v \in \V_{G/H}$. The following statement shows how to detect the units in the associated graded algebras $\gr_v(\k[X_B])$ using the $(\k^* \times G)$-algebra $\gr_{\tilde{v}}(A)$. 
\begin{Prop}   \label{prop-unit-gr_v-gr-v-tilde}
Let $h \in \k[X_B]$ and let $f_k \in A_k$, for some $k>0$, be such that $\pi(f_k) = f_k / s^k = h$. Then $\In_v(h) \in \gr_v(\k[X_B])$ is a unit if and only if $\In_{\tilde{v}}(f_k)$ divides a power of $\In_{\tilde{v}}(s)$.
\end{Prop}
\begin{proof}
Suppose $\In_v(h) \in \gr_v(\k[X_B])$ is a unit. Then there exists $h' \in \k[X_B]$ such that $v(hh') = v(1) = 0$ and $v(hh' - 1) > 0$. Since $\pi: A \to \k[X_B]$ is surjective we can find 
$f'_\ell \in A_\ell$, for some $\ell > 0$, with $\pi(f'_\ell) = h'$. Then $\tilde{v}(f_kf'_\ell) = v(hh') = 0$ and $\tilde{v}(f_k f'_\ell - s^{k+\ell}) = v(hh' -1) > 0$. Note that $\tilde{v}(s^{k+\ell}) = v(1) = 0$.
This shows that $\In_{\tilde{v}}(f_k) ~\In_{\tilde{v}}(f'_\ell) = \In_{\tilde{v}}(s)^{k+\ell}$ in the algebra $\gr_{\tilde{v}}(A)$. Conversely, suppose $\In_{\tilde{v}}(f_k)$ divides a power of $\In_{\tilde{v}}(s)$. Since $\In_{\tilde{v}}(s)$ and $\In_{\tilde{v}}(f_k)$ are homogeneous, there is $f'_\ell \in A_\ell$, for some $\ell > 0$, such that $\tilde{v}(f_kf'_\ell) = \tilde{v}(s^{k+\ell})$ and $\tilde{v}(f_k f'_\ell - s^{k+\ell}) > 0$. As before, this implies that $v(hh') = 0$ and 
$v(hh' - 1) > 0$ where $h' = f'_\ell / s^\ell$. This means that $\In_v(h)$ is a unit in $\gr_v(\k[X_B])$ as required.  
\end{proof}

Now since $A$ and hence $\gr_{\tilde{v}}(A)$ are $(\k^* \times G)$-algebras, we can give a criterion for when the ideal generated by an element in $\gr_{\tilde{v}}(A)$ 
contains a power of $\In_{\tilde{v}}(s)$, in terms of its isotypic decomposition.
Take $f \in A$ and let us write $f = \sum_{k, \lambda} f_{k, \lambda}$ where $f_{k, \lambda}$ is the $(k, \lambda)$-isotypic component of $f$ in $A$, i.e. $f_{k, \lambda}$ lies in the 
$\lambda$-isotypic component of $A_k$.
\begin{Prop}  \label{prop-conditions-1-2}
Let $\tilde{v} \in \tilde{\V}_A$. Then $\In_{\tilde{v}}(f)$ divides a power of $\In_{\tilde{v}}(s)$ if and only if the following conditions are satisfied:
\begin{itemize}
\item[(1)] The minimum $\min\{ \langle \tilde{v}, (k, \lambda)\rangle \mid f_{k, \lambda} \neq 0 \}$ is attained at a unique $(k_0, \lambda_0)$.
\item[(2)] Moreover, there exists $(\ell_0, \mu_0) \in \tilde{\Lambda}^+_A$ and $f'_{\ell_0, \mu_0} \in A_{\ell_0, \mu_0}$ such that 
$\In_{\tilde{v}}(f_{k_0, \lambda_0}) \In_{\tilde{v}}(f'_{\ell_0, \mu_0}) = \In_{\tilde{v}}(s^{k_0 + \ell_0})$. In particular, $(k_0, \lambda_0)$ lies in the semigroup generated by $-\tilde{\Lambda}^+_A$ and $(1, \theta)$, where $\theta$ is the $B$-weight of the section $s$. 
\end{itemize} 
\end{Prop}
\begin{proof}
As in Section \ref{sec-spherical-Grobner-fan}, let $\succ$ be a total order refining the spherical dominant order for $A$ regarded as a $(\k^* \times G)$-algebra. We then consider the 
total order $\succ_{\tilde{v}}$ and the associated graded algebra $\gr_{\succ_{\tilde{v}}}(A)$ which is a $\tilde{\Lambda}^+_A$-graded algebra.
Now let us assume that $\In_{\tilde{v}}(f)$ divides a power of $\In_{\tilde{v}}(s)$. Thus there is $f' \in A$ and $m>0$ such that 
\begin{equation}  \label{equ-f-f'-s}
\In_{\tilde{v}}(f)~ \In_{\tilde{v}}(f') = \In_{\tilde{v}}(s)^m. 
\end{equation}
Let $f' = \sum_{\ell, \mu} f'_{\ell, \mu}$ be the isotypic decomposition of $f'$. 
We note that $s$ is a $(\k^* \times B)$-eigensection with weight $(1, \theta)$ and hence is 
its own isotypic decomposition. Let $\bar{f}$, $\bar{f'}$ and $\bar{s}$ be the 
elements in $\gr_{\succ_{\tilde{v}}}(A)$ represented by the same isotypic decompositions as $\In_{\tilde{v}}(f)$, $\In_{\tilde{v}}(f')$ and $s$ respectively.  
From \eqref{equ-f-f'-s} we know that $\bar{f}\bar{f'} = \bar{s}^m$. But the algebra $\gr_{\succ_{\tilde{v}}}(A)$ is a $\tilde{\Lambda}^+_A$-graded algebra and 
$\bar{s}^m$ is a homogeneous element in this grading. It follows that $\bar{f}$ and $\bar{f'}$ are also $\tilde{\Lambda}^+_A$-homogeneous (because in a graded domain 
the product of a non-homogeneous element with any other element cannot be homogeneous). This shows that $\bar{f}$ and $\bar{f'}$ each consist of a single isotypic component.
This readily implies the conditions (1) and (2). Conversely, assume that the conditions (1) and (2) are satisfied. One then verifies that 
$\In_{\tilde{v}}(f)~ \In_{\tilde{v}}(f'_{\ell_0, \mu_0})$ is equal to $\In_{\tilde{v}}(s)^m$ in the algebra $\gr_{\tilde{v}}(A)$, where $m = k_0 + \ell_0$. This finishes the proof.
\end{proof}

\begin{Rem}  \label{rem-conditions-piecewise-lin}
We observe that given $f \in A$, the conditions (1) and (2) in Proposition \ref{prop-conditions-1-2} are piecewise linear conditions on $v$. In fact, for each $(k_0, \lambda_0)$ appearing in $f$ the condition that 
$\min\{\langle \tilde{v}, (k, \lambda)\rangle \mid f_{k, \lambda} \neq 0 \}$ is attained only at $(k_0, \lambda_0)$, is a piecewise linear condition on $\tilde{v}$. Also (2) says that we want the minimum in (1) to be attained only at those $(k_0, \lambda_0)$ which satisfy a certain condition as stated in (2). On the other hand, by Proposition \ref{prop-tilde-v-valuation} we know that the map $v \mapsto \tilde{v}$ is a linear map from $\V_{G/H} \to \tilde{\V}_A$. This shows that (1) and (2) together impose a piecewise linear condition on $v$. 
\end{Rem}

\subsection{Spherical tropical bases} \label{subsec-spherical-trop-basis}
In this section we show that any ideal $J \subset \k[X_B]$ possesses a finite spherical tropical basis (as defined below). In other words, the spherical tropical variety $\trop_B(J)$ is an intersection of a finite number of spherical tropical hypersurfaces. 



\begin{Def}[Spherical tropical basis]   \label{def-sph-trop-basis}
Let $J \subset \k[X_B]$ be an ideal. A set $\mathcal{T}=\mathcal{T}(J) \subset J$ is a {\it spherical tropical basis} for $J$ if for every $v \in \V_{G/H}$ the following holds: $v \in \trop_B(J)$, i.e.
$\In_v(J) \subset \gr_v(\k[X_B])$ does not contain a unit element, if and only if for all $f \in \mathcal{T}$, $\In_v(f)$ is not a unit element.
\end{Def}

\begin{Th}[Existence of a finite spherical tropical basis] \label{th-sph-trop-basis}
Every ideal $J \subset \k[X_B]$ has a finite spherical tropical basis.
\end{Th}

\begin{proof}
Let $J \subset \k[X_B]$ be an ideal. The strategy of the proof is as follows. Let $v$ be a $G$-invariant valuation which lies in $\trop_B(J)$. This means that 
the initial ideal $\In_v(J) \subset \gr_v(\k[X_B])$ contains a unit element. As before let $\tilde{v}$ be the valuation on the graded algebra $A$ associated to $v$ (see \eqref{equ-tilde-v}). Also let $I \subset A$ be the homogenization of $J$, i.e. the ideal generated by homogeneous elements in $\pi^{-1}(J)$.  
We construct a finite spherical tropical basis for $J$ using the fact that there are finite number of cones in 
the spherical Gr\"obner fan of $I$ (Theorem \ref{th-sph-Grobner-fan}). The valuation $\tilde{v}$ lies in the relative interior of some cone $\sigma$ 
in the spherical Gr\"obner fan of $I$. Now we would construct an element $F_{\sigma} \in J$ such that the following holds: 
for each valuation $v' \in \V_{G/H}$ for which $\tilde{v}'$ lies in the relative interior of the cone $\sigma$, the element $\In_{v'}(F_\sigma)$ is a unit in $\gr_{v'}(\k[X_B])$.
To get a spherical tropical basis $\mathcal{T}$ it then suffices to define:
$$\mathcal{T} = \{ F_{\sigma} \mid \sigma \textup{ is a cone in } \textup{GF}(I) \textup{ that intersects the image of } \trop_B(J) \}.$$
In the rest of the proof we explain how to construct $F_{\sigma} \in J$ with the desired property.

Let $\tilde{v}' = \tilde{v} + \epsilon \tilde{w}$
be a valuation in $\tilde{\V}_{A}$ for $\tilde{w}$ in the linear span of $\tilde{\V}_{A}$ and  $\epsilon > 0$. By Proposition \ref{prop-generic-initial-ideal-is-monomial}, 
if $\tilde{w}$ is generic and $\epsilon$ is sufficiently small, then the graded algebra $\gr_{\tilde{v}'}(A)$ is the horospherical contraction $A_{\hc}$ of $A$ and 
the initial ideal $\In_{\tilde{v}'}(I)$ is a $\tilde{\Lambda}_{A}^+$-homogeneous ideal. That is, 
if $q = \sum_{i, \lambda} q_{i, \lambda}$ is the isotypic decomposition of an element $q \in \In_{\tilde{v}'}(I)$ then $q_{i, \lambda} \in \In_{\tilde{v}'}(I)$
for all $i$, $\lambda$. 

Now let $h \in J$ be such that $\In_v(h)$ is a unit element in $\gr_v(\k[X_B])$. Let $f \in A_k$, for some $k>0$, be such that $f / s^k = h$. 
Let us write $f = \sum_\lambda f_{\lambda}$ as sum of its isotypic components. 
By Proposition \ref{prop-unit-gr_v-gr-v-tilde} we know that $f$ satisfies the conditions (1) and (2) in Proposition \ref{prop-conditions-1-2}.
Let $f_{\lambda_0}$ be the isotypic component of $f$ where the unique minimum $\min\{ \langle \tilde{v}, (k, \lambda) \rangle \mid f_\lambda \neq 0 \}$ is attained. 

We can find $F' \in A$ such that $F = f_{\lambda_0} - F' \in I$ {and no isotypic component of $F'$ lies in the generic initial ideal $\In_{\tilde{v}'}(I)$}. We claim that for any valuation $v'' \in \V_{G/H}$ such that $\tilde{v}''$ lies in the relative interior of $\sigma$ we have 
$\In_{\tilde{v}''}(F) = f_{\lambda_0}$. First we note that since $\In_{\tilde{v}''}(I) = \In_{\tilde{v}'}(I)$, the initial element $\In_{\tilde{v}''}(F)$ can 
be regarded as an element of $\In_{\tilde{v}'}(I)$. Now if $\In_{\tilde{v}''}(F) \neq f_{\lambda_0}$ then it means that a nonzero sum of isotypic components of 
$F'$ lies in $\In_{\tilde{v}'}(I)$. But this initial ideal is $\tilde{\Lambda}_{A}^+$-homogeneous and hence it follows that at least some of the isotypic components of $F'$ are in $\In_{\tilde{v}'}(I)$ 
which contradicts the choice of $F'$. Thus we conclude that $\In_{\tilde{v}''}(F) = f_{\lambda_0}$. 
Finally, combining Propositions \ref{prop-conditions-1-2} and \ref{prop-unit-gr_v-gr-v-tilde} (applied to $F$ and the valuation $\tilde{v}''$) 
we see that  that $\pi(F) \in \In_{v''}(J)$ is a unit element in $\gr_{v''}(J)$ as desired.
\end{proof}

\begin{Cor} \label{cor-sph-trop-basis}
The spherical tropical variety of an ideal is a finite intersection of spherical tropical hypersurfaces. 
\end{Cor}


\subsection{Spherical tropical variety of a subscheme in $G/H$}  \label{subsec-sph-trop-var-G/H}

Finally, we define the spherical tropical variety of a subscheme in the spherical homogeneous space $G/H$. As before, let $L$ be a very ample $G$-linearized line bundle on a projective spherical embedding $\overline{X}$ of $G/H$. Let $A = \bigoplus_{i \geq 0} H^0(\overline{X}, L^{\otimes i})$ be the algebra of sections of $L$. 

\begin{Def}[Spherical tropical variety of an ideal in an algebra of sections]   \label{def-sph-trop-var-G/H}
Let $I \subset A$ be an ideal. We define the {\it spherical tropical variety of $I$} to be the set of all $v \in \V_{G/H}$ for which there exists a Borel subgroup $B$ (depending on $v$) such that $\In_{\tilde{v}}(I)$ does not contain any $B$-eigensections.
\end{Def}

Proposition \ref{prop-sph-trop-var-G/H-vs-X_B} below shows how $\trop(I)$ encodes the behavior at infinity of the subscheme defined by $I$, and moreover, how it is related to the tropical varieties of ideals in coordinate rings of Borel open orbits. 

\begin{Rem}  \label{rem-B-eigensectio-like-monomial}
In the spherical setting, we regard the notion of a $B$-eigensection as a generalization of the notion of a monomial from the classical torus case.  
\end{Rem}

Let $Y \subset G/H$ be a subvariety and let $I = I(Y)$ be the ideal of sections in $A$ that vanish on $Y$. Also for a Borel subgroup $B$ let $J_B$ denote the ideal of regular functions in $\k[X_B]$ vanishing on $Y \cap X_B$. Recall that $X_v$ denotes the spherical embedding corresponding to the single ray generated by $v$. The unique $G$-stable divisor in $X_v$ is denoted by $D_v$.
\begin{Prop} \label{prop-sph-trop-var-G/H-vs-X_B}
With notation as above, $v \in \trop(I)$ if and only if the closure of $Y$ in $X_v$ intersects $D_v$. Moreover, we have the following:
\begin{equation} \label{equ-sph-trop-var}
\trop(I) = \bigcup_B \trop_B(J_B),
\end{equation}
where the union is over all the Borel subgroups $B \subset G$. (In Proposition \ref{prop-sph-trop-var-finite-B} below we show that 
there is a finite collection of Borel subgroups, independent of the choice of $Y$, which suffice for defining the righthand side of \eqref{equ-sph-trop-var}.)
\end{Prop}
\begin{proof}
Let $v \in \V_{G/H}$. We note that since by assumption $s$ vanishes on all the $B$-stable divisors in $\overline{X}$, the ideal $\In_{\tilde{v}}(I)$ contains a $B$-eigensection if and only if it contains a power of $\In_{\tilde{v}}(s)$. The equality \eqref{equ-sph-trop-var} now follows from Proposition \ref{prop-unit-gr_v-gr-v-tilde}. It remains to prove the first assertion. For $v \in \V_{G/H}$ let $\overline{Y}$ denote the closure of $Y$ in $X_v$. Suppose $v \in \trop(I)$. Then, by Proposition \ref{prop-unit-gr_v-gr-v-tilde}, there exists a Borel subgroup $B$ such that 
$v \in \trop_B(J_B)$. Theorem \ref{th-sph-trop-var-X_v_B} then shows that the closure of 
$Y \cap X_B$ in $X_{v, B}$ intersects $D_v'$. This readily implies that $\overline{Y}$ intersects $D_v$. 
Conversely, let $x \in \overline{Y} \cap D_v$. Clearly there exists a Borel subgroup $B$ such that $x$ lies 
in the open $B$-orbit $D_v' \subset D_v$. Consider the affine open subset $X_{v, B} \subset X_v$. It contains the open $B$-orbit $X_B$.
Since $X_{v, B}$ is open, $x$ lies in the closure of $Y \cap X_B$ in $X_{v, B}$. Again from Theorem \ref{th-sph-trop-var-X_v_B} we see that $v \in \trop_B(J_B)$  and hence $v \in \trop(Y)$.
\end{proof}

Next, we show that the union in the righthand side of \eqref{equ-sph-trop-var} is finite. The proof uses the Luna-Vust theory of spherical embeddings. 
\begin{Prop} \label{prop-sph-trop-var-finite-B}
With notation as above, there exists a finite number of Borel subgroups $B_1, \ldots, B_s$ such that for any subscheme $Y$ in $G/H$ 
we have:
$$\trop(I) = \bigcup_{i=1}^s \trop_{B_i}(J_{B_i}),$$
where $I = I(Y) \subset A$ is the ideal of sections vanishing on $Y$, and $J_{B_i} \subset \k[X_{B_i}]$ is the ideal of $Y \cap X_{B_i}$.
\end{Prop}
\begin{proof}
Let $X_\Sigma$ be a complete spherical embedding of $G/H$ corresponding to a complete fan $\Sigma$. 
Choose a finite collection of Borel subgroups $B_1, \ldots, B_s$ such that for each $G$-orbit $O$ in $X_\Sigma$
the open orbits in $O$ corresponding to $B_1, \ldots, B_s$ cover the whole $O$.
Now take a valuation $v \in \V_{G/H}$. Since $\Sigma$ is a complete fan, there exists a unique cone $\sigma \in \Sigma$ such that $v$ lies in the relative interior of 
$\sigma$. The cone $\sigma$ corresponds to a simple spherical subvariety $X_\sigma \subset X_\Sigma$. Let $O_\sigma$ denote the unique closed $G$-orbit 
in $X_\sigma$ (note that $O_\sigma$ may not be closed in the larger variety $X_\Sigma$). 
By the Luna-Vust theorey (\cite[Theorem 4.1]{Knop-LV}), the inclusion $v \in \sigma$ gives a $G$-equivariant morphism $\psi: X_v \to X_\sigma$ and $\psi$ maps 
$D_v$ to $O_\sigma$. 
Let $\overline{Y}$ be the closure of $Y$ in $X_v$. 
Suppose that the scheme-theoretic intersection 
$\overline{Y} \cap D_v$ is nonempty and let $x$ be a closed point 
in $\overline{Y} \cap D_v$. Then $\psi(x)$ lies in $O_\sigma$. We know that there is $i$ such that $\psi(x)$ 
lies in the open $B_i$-orbit in $O_\sigma$. Since there are only a finite number of $B_i$-orbits we conclude that $x$ itself lies in the open $B_i$-orbit 
in $D_v$. This finishes the proof.
\end{proof}

\section{Spherical tropical varieties via tropicalization map} \label{sec-sph-trop-var-trop-map}
In this section we discuss construction of the spherical tropical variety via the spherical tropicalization map.

\subsection{Germs of curves and spherical tropicalization map} \label{subsec-tropicalization}
We start by recalling the notions of a germ of a curve and a formal curve (see for example \cite[Section 24]{Timashev}).
As usual $\k$ denotes the ground field. 
We let $\mathcal{O} = \k[[t]]$ denote the algebra of formal power series with coefficients in $\k$ and 
$\K = \k((t))$ its field of fractions, i.e. the field of formal Laurent series. If $f \in \K$ we denote by $\ord_t(f)$ the order of $t$ in the Laurent series $f$. 
Clearly $\ord_t$ is a $\z$-valued valuation on the field $\K$.

Let $X$ be a variety. A {\it germ of an algebraic curve} or simply a {\it germ of a curve} on $X$ is a pair $(\gamma, t_0)$ where $\gamma$ is a rational map 
from a smooth projective curve $\Gamma$ to $X$ and $t_0 \in \Gamma$ is a base point. One says $(\gamma, \theta_0)$ is {\it convergent} if $\gamma$ is regular at $\theta_0$. One can think of $\gamma$ as a point in $X(\k(\Gamma))$, i.e. a $\k(\Gamma)$-point on $X$. A {\it germ of a formal curve} or simply a {\it formal curve} $\gamma$ on $X$ is a $\K$-point of $X$. An $\mathcal{O}$-point on $X$ is called a {\it convergent formal curve}. The limit of a convergent formal curve is the point on $X(\k)$ obtained by setting $t=0$ in $\gamma$. 
It is natural to think of a formal curve as a parameterized analytic curve in $X$. If we assume $X$ is sitting in an affine space $\A^N$ 
then a formal curve $\gamma$ on $X$ is an $N$-tuple of Laurent series satisfying the defining equations of $X$ in $\A^N$. If $\gamma$ is convergent then
its coordinates are power series and their constant terms are the coordinates of the limit point $\gamma_0 = \lim_{t \to 0} \gamma(t)$. 

If $(\gamma, \theta_0)$ is a germ of a curve on $X$ then the inclusion $\k(\Gamma) \subset \K$ 
shows that $\gamma$ gives a formal curve on $X$ (this depends on the choice of a formal uniforming parameter $t$ at the complete local ring
$\hat{\O}_{\Gamma, \theta_0}$). Also the inclusions $\O_{\Gamma, \theta_0} \subset \hat{\O}_{\Gamma, \theta_0} 
\cong \mathcal{O}$ show that if $(\gamma, \theta_0)$ is a convergent germ of a curve then the corresponding formal curve is also convergent. 
By abuse of notation we will denote the formal curve associated to a germ of a curve $(\gamma, \theta_0)$ again by $\gamma$.

\begin{Def}[Valuation associated to a formal curve] \label{def-val-formal-curve}
A formal curve $\gamma$ on $X$ defines a valuation $v_\gamma: \k(X) \to \z \cup \{\infty\}$ as follows. 
\begin{equation} \label{equ-val-formal-curve}
v_\gamma(f) = \ord_t(f(\gamma(t))).
\end{equation}
\end{Def}

There is a $t$-adic topology on $X(\K)$, the set of formal curves on $X$, which is thinner than the Zariski topology. For $X = \A^N$ a basic $t$-adic 
neighborhood of a point $\gamma = (\gamma_1, \ldots, \gamma_N)$ consists of all $\tau = (\tau_1, \ldots, \tau_N)$ such that 
$$\ord_t(\gamma_i - \tau_i) \geq C, \quad \forall i = 1, \ldots, N,$$ where $C \in \n$.
The $t$-adic topology on arbitrary varieties is induced from that on the affine space using affine charts. 
The following basic result due to Michael Artin says that formal curves on a variety $X$ can be approximated by germs of algebraic curves arbitrarily 
closely in $t$-adic topology.
\begin{Th}[Artin approximation theorem] \label{th-Artin-approx}
The set of formal germs induced by germs of algebraic curves on $X$ is dense in $X(\K)$ with respect to the $t$-adic topology.
\end{Th}

From the Artin approximation theorem (Theorem \ref{th-Artin-approx}) the following readily follows:
\begin{Cor} \label{cor-Artin-approx}
Let $\gamma \in Y(\K)$ be a formal curve on $Y$. Also let $f_1, \ldots, f_s \in \mathcal{O}(Y)$ be a finite number of regular functions on $Y$ and 
$C > 0$ a constant. Then there exists a germ of algebraic curves $\gamma'$ on $Y$ such that:
$$\ord_t(f_i(\gamma(t)) - f_i(\gamma'(t))) > C, \quad \forall i = 1, \ldots, s.$$ 
\end{Cor}

We recall that the algebraic closure $\overline{\K}$ of the field of formal Laurent series is the field of Puiseux series (see Section \ref{subsec-tropical-prelim}).
We call a point in $X(\overline{\K})$ a {\it formal Puiseux curve} or simply a {\it Puiseux curve} on $X$. The definition of valuation associated to a curve extends to 
Puiseux curve as well. That is, a formal Puiseux curve $\gamma$ on $X$ gives a valuation $v_\gamma: \k(X) \to \q \cup \{\infty\}$, defined by the same equation 
\eqref{equ-val-formal-curve}.

Now we turn to the case of spherical varieties and $G$-invariant valuations. As usual we let $G/H$ be a spherical homogeneous space.
The main construction in spherical tropicalization is the construction of a $G$-invariant valuation from a given valuation on ${G/H}$.
The following well-known result is the key to this construction (see \cite[Lemma 1.4]{Knop}, \cite[Lemma 10 and 11]{Sumihiro}, \cite[3.2 Lemme]{Luna-Vust}).

\begin{Th}[Sumihiro] \label{th-Sumihiro}
Let $v: \k({G/H}) \setminus \{0\} \to \r$ be any valuation.
\begin{itemize}
\item[(1)] For every $0 \neq f \in \k({G/H})$ there exists a nonempty Zariski open subset $U_f \subset G$ such that the value $v(g \cdot f)$ is the same 
for all $g \in U_f$. Let us denote this value by $\bar{v}(f)$, i.e.
$$\bar{v}(f) = v(g \cdot f), \quad \forall g \in U_f.$$
\item[(2)] We have $\bar{v}(f) = \min\{ v(g \cdot f) \mid g \in G \}$.
\item[(3)] $\bar{v}$ is a $G$-invariant valuation on ${G/H}$. 
\end{itemize}
\end{Th}

Recall that a formal curve $\gamma$ on ${G/H}$ gives rise to a valuation $v_\gamma$. 
We let $\bar{v}_\gamma$ denote the $G$-invariant valuation defined by:
$$\bar{v}_\gamma(f) = \ord_t(f(g \cdot \gamma(t))),$$ for every $0 \neq f \in A$ and $g \in G$ in general position.
Following \cite{Vogiannou}, we call the map $$\Trop: G/H(\overline{\K}) \to \V_{G/H}, \quad \gamma \mapsto \bar{v}_\gamma,$$ the {\it spherical tropicalization map}.

\begin{Ex} \label{ex-SL(2)-inv-val}
As in Example \ref{ex-sph-var}(3) consider the spherical variety $X = \A^2 \setminus \{(0,0)\}$ for the natural action of $G = \SL(2, \k)$.
The algebra of regular functions $\k[X]$ is just the polynomial ring $\k[x, y]$.
The weight lattice $\Lambda_X$ coincides with the weight lattice $\Lambda$ of $G$ and can be identified with $\z$. The function $f(x, y) = y$ is a $B$-eigenfunction in $\k[X]$
whose weight is $1$, namely the generator of $\Lambda_X$. Let $\gamma = (\gamma_1, \gamma_2)$ be a formal curve in $X = \A^2 \setminus \{0\}$. Let us write 
$\gamma_1(t) = \sum_i a_i t^i$ and $\gamma_2(t) = \sum_i b_i t^i$. Let $g = \left[ \begin{matrix} g_{11} & g_{12} \\ g_{21} & g_{22} \end{matrix} \right]$. 
We compute that $f(g \cdot \gamma(t)) = g_{21} \gamma_1 + g_{22} \gamma_2$. 
From the definition of the $G$-invariant valuation $\bar{v}_\gamma$ we have $\bar{v}_\gamma(y) = \ord_t(g \cdot \gamma(t))$ for $g$ in general position. Thus
\begin{equation} \label{equ-v-gamma-A^2}
\bar{v}_\gamma(y) = \ord_t(g_{21} \gamma_1(t) + g_{22} \gamma_2(t)) = \min(\ord_t(\gamma_1(t)), \ord_t(\gamma_2(t)).
\end{equation}
\end{Ex}

There is another way of understanding the $G$-invariant valuation associated to a formal curve and that is through the generalized Cartan decomposition for spherical varieties.
It goes back to Luna and Vust (\cite{Luna-Vust}). 
A proof of it can also be found in \cite[Theorem 8.2.9]{Gaitsgory-Nadler}.
\begin{Th}[Generalized non-Archimedean Cartan decomposition for spherical varieties over $\K$]  \label{th-non-Arch-Cartan-decomp} 
The $G(\O)$-orbits in ${G/H}(\K)$ are parameterized by $\check{\Lambda}_{G/H} \cap \V_{G/H}$. Here 
$\check{\Lambda}_{G/H} \subset \mathcal{Q}_{G/H}$ denotes the lattice dual to the weight lattice $\Lambda_{G/H}$, and a one-parameter subgroup 
$\lambda \in \check{\Lambda}_{G/H}$ 
corresponds to the orbit through the formal curve $\lambda(t) \in T_{G/H}(\K)$.
\end{Th}

{Thus the valuation $\bar{v}_\gamma$ can be interpreted as the valuation represented by the point of intersection of the $G(\mathcal{O})$-obit of $\gamma$ in 
${G/H}(\K)$ and the image of valuation cone $\V_{G/H}$ (under the exponential map) in ${G/H}(\K)$.}  

\begin{Ex}[non-Archimedean Cartan decomposition] \label{ex-non-Arch-Cartan-decomp}
As in Example \ref{ex-sph-var}(5) consider $G$ with left-right action of $G \times G$. Theorem \ref{th-non-Arch-Cartan-decomp} applied in this case recovers the
a non-Archimedean version of the usual Cartan decomposition.(see \cite{Iwahori}). With notation as above, it states that:
$$G(\K) = G(\O) \cdot \check{\Lambda}^+ \cdot G(\O).$$
{Here $\check{\Lambda}$ is the dual lattice to the weight lattice $\Lambda$ and $\check{\Lambda}^+$ is the intersection of $\check{\Lambda}$ with the dual positive Weyl chamber. We regard both as subsets of $T(\K)$.} 

When $G = \GL(n, \c)$ the above non-Archimedean Cartan decomposition gives the well-known Smith normal form of a matrix (over the 
field of formal Laurent series $\K$ which is the field of fractions of the principal ideal domain $\O$, the ring of formal power series).

\end{Ex}

\begin{Ex}[Non-Archimdean Iwasawa decomposition] \label{ex-non-Arch-Iwasawa}
As in Example \ref{ex-sph-var}(4) consider the spherical homogeneous space $G/U$ where $U$ is a maximal unipotent subgroup of $G$. In this case 
Theorem \ref{th-non-Arch-Cartan-decomp} gives a non-Archimdean version of the Iwasawa decompostion (see \cite{Iwahori}). It states that: 
$$G(\K) = G(\O) \cdot \check{\Lambda} \cdot U(\K),$$
where as in the previous example, $\check{\Lambda}$ is the dual lattice to the weight lattice $\Lambda$ and we regard it as a subset of $T(\K)$.
\end{Ex}

\begin{Rem} \label{rem-Arch-Cartan-decomp}
In Section \ref{sec-sph-amoeba} we will interpret the usual (Archimedean) Cartan decomposition and the Iwasawa decomposition as giving us a 
spherical generalization of the notion of an amobea of a subvariety (see also Section \ref{subsec-tropical-prelim}). 
Using this point of view, we will make a connection between the Archimedean 
and non-Archimedean Cartan decompositions. More precisely, we see that the non-Archimedean Cartan decomposition can be interpreted as a ``limit'' of the 
Archimedean Cartan decomposition.
\end{Rem}

Finally, we come to the main definition of this section which is the spherical tropicalization of a subvariety.
\begin{Def}(Spherical tropicalization) \label{def-sph-tropicalization}
Let $Y \subset {G/H}$ be a subvariety. Following \cite{Vogiannou} we call the image $\Trop(Y(\overline{\K})) \subset \V_{G/H}$
the {\it spherical tropicalization of $Y$}. 
\end{Def}

In \cite[Theorem 1.2]{Vogiannou}, it is proved that $\Trop(Y)$ coincides with the support of a rational polyhedral fan.
The content of our Fundamental Theorem (Theorem \ref{th-fundamental}) is that this $\Trop$ construction coincides with $\trop$ construction in 
Section \ref{subsec-sph-Grobner} using initial ideals. 

The following key fact is proved in \cite[Proposition 4.5]{Vogiannou}. It shows that $\Trop(Y)$ encodes the asymptotic behavior of the subvariety $Y$ in all possible spherical embeddings of 
$G/H$. As before, $X_v$ denotes the spherical embedding of ${G/H}$ associated to $v \in \V_{G/H}$ and $D_v$ is the unique $G$-stable prime divisor in $X_v$.
\begin{Prop} \label{prop-Vogiannou}
Let $Y \subset {G/H}$ be a subvariety. 
A valuation $v \in \V_{G/H}$ belongs to the spherical tropical variety $\Trop(Y)$ if and only if the closure $\overline{Y}$ of $Y$ in $X_v$ 
intersects the unique closed $G$-orbit $D_v$.
\end{Prop}

To wrap up this section, we prove a theorem about approximation by germs of algebraic curves. It is a corollary of the Artin approximation theorem stated above. 
It implies that in defining the tropical variety one can assume that $\gamma$ is an algebraic curve. While this statement is interesting on its own, 
it will be also useful later in the proof of Theorem 
\ref{th-amoeba-trop-var-curve}.

\begin{Th}[Approximation by algebraic points] \label{th-approx-algebraic}
Let $Y \subset {G/H}$ be a subvariety. Let $\gamma$ be a point of $Y(\K)$, i.e. a formal curve in $Y$. Then there exists a germ of an 
algebraic curve $\gamma'$ on $Y$ such that: $$\bar{v}_\gamma = \bar{v}_{\gamma'},$$ 
i.e. $\gamma$ and $\gamma'$ give rise to the same invariant valuation in $\V_{G/H}$.
\end{Th}

We need the following easy lemma in the proof.
\begin{Lem} \label{lem-easy}
Let $p(t)$ be a formal Lauren series. Pick a constant $C \geq \ord_t(p)$ and suppose that for 
some other formal Laurent series $q(t) \in \K$ we have $\ord_t(p - q) > C$. Then $\ord_t(p) = \ord_t(q)$. 
\end{Lem}
\begin{proof}
By contradiction suppose that $\ord_t(p) \neq \ord_t(q)$. Then by the non-Archimedean property of $\ord_t$ we have 
$\ord_t(p - q) = \min(\ord_t(p), \ord_t(q))$. But this is impossible since by assumption the lefthand side is bigger than 
$\ord_t(p) \geq \min(\ord_t(p), \ord_t(q))$.
\end{proof}

\begin{proof}[Proof of Theorem ]
Let $h_1, \ldots h_\ell \in A$ be $B$-weight functions whose weights generate the lattice $\Lambda_{G/H}$. Recall 
that any $G$-invariant valuation is determined by its values on the $h_i$, that is, if $\bar{v}_\gamma(h_i) = \bar{v}_{\gamma'}(h_i)$ for all 
$i=1, \ldots, \ell$ then $\bar{v}_\gamma = \bar{v}_{\gamma'}$ (Proposition \ref{prop-inv-val-min-isotypic}). 

Take $1 \leq i \leq \ell$ and let $M_i$ be the $G$-module generted by $h_i$. Since $A$ is a rational $G$-module, $M_i$ is a finite dimensional 
vector space which is spanned by the $g \cdot h_i$, $g \in G$. Thus we can find a vector space basis for $M_i$ of the form 
$\{g_{i1} \cdot h_i, \ldots, g_{is_i} \cdot h_i\}$ where $g_{i1}, \ldots, g_{is_i} \in G$. Now for any $g \in G$ we can write 
$g \cdot h_i = \sum_{j=1}^{s_i} c_{ij}(g) (g_{ij} \cdot h_i)$. We thus have:
$$v_\gamma(g \cdot h_i) \geq \min\{v_\gamma(g_{ij} \cdot h_i) \mid c_{ij}(g) \neq 0\}.$$
From this we see that for generic $g \in G$ we have:
\begin{equation} \label{equ-v-gamma-g-ij}
\bar{v}_\gamma(h_i) := v_\gamma(g \cdot h_i) = \min\{v_\gamma(g_{ij} \cdot h_i) \mid c_{ij}(g) \neq 0\}.
\end{equation}
Now take a constant $C$ which is greater than all the $v_\gamma(g_{ij} \cdot h_i)$ for all $i, j$. 
By Corollary \ref{cor-Artin-approx} (Artin approximation) we can find a germ of algebraic curves $\gamma'$ on $Y$ such that for all $i, j$ we have:
$$\ord_t( v_\gamma(g_{ij} \cdot h_i) - v_{\gamma'}(g_{ij} \cdot h_i) ) > C.$$
But Lemma \ref{lem-easy} then implies that $\ord_t(v_\gamma(g_{ij} \cdot h_i)) = \ord_t(v_{\gamma'}(g_{ij} \cdot h_i))$ for all $i, j$.
Using \eqref{equ-v-gamma-g-ij} this gives us that $\bar{v}_\gamma(h_i) = \bar{v}_{\gamma'}(h_i)$ for all $i$ and hence 
$\bar{v}_\gamma = \bar{v}_{\gamma'}$ as required.
\end{proof}

The following is an immediate corollary of Theorem \ref{th-approx-algebraic}.
\begin{Cor} \label{cor-trop-var-alg-curves}
In the definition of spherical tropical variety it suffices to use germ of algebraic curves. In other words, with notation as above, the two sets 
${\Trop(Y)}$ and $\{ \Trop(\gamma) \mid \gamma \textup{ is a germ of an algebraic curve on } Y \}$ coincide.
\end{Cor}

\subsection{A fundamental theorem for spherical tropical geometry}   \label{subsec-fundamental}
Finally we formulate a generalization of the fundamental theorem of tropical geometry to the spherical setting. It states that all the different constructions of the spherical tropical variety we discussed coincide. Namely: 
(1) the construction using initial ideals and Borel charts (Definition \ref{def-sph-trop-var-X_B}), (2) the construction using initial ideals in algebra of sections of a line bundle (Definition \ref{def-sph-trop-var-G/H}), and (3) the construction using the tropicalization map and formal curves (Definition \ref{def-sph-tropicalization}). It is an immediate corollary of the results discussed in the previous sections. 

Let $Y \subset G/H$ be a subvariety. As before take a $G$-linearized very ample line bundle $L$ on a projective spherical embedding $\overline{X}$ of $G/H$ and let $A = A(\overline{X}, L)$ denote its algebra of sections with $I = I(Y) \subset A$ the ideal of sections vanishing on $Y$. Also for each Borel subgroup $B$, let $J_B \subset \k[X_B]$ be the defining ideal of $Y$ intersected with the open $B$-orbit $X_B$. 
\begin{Th}[Fundamental theorem] \label{th-fundamental}
With notation as above, the following sets coincide:
\begin{itemize}
\item[(a)] $\trop(I) = \{ v \in \V_{G/H} \mid \In_{\tilde{v}}(I) \textup{ does not contain any $B$-eigensection for some Borel $B$}\}.$ (see Section \ref{subsec-sph-trop-var-G/H}).
\item[(b)] $\bigcup_B \trop_B(J_B),$ where the union is over all Borel subgroups of $G$ (recall that by Proposition \ref{prop-sph-trop-var-finite-B} it is enough to take the union over a finite collection of Borel subgroups).
\item[(c)] $\Trop(Y) = \{\Trop(\gamma) \in \V_{G/H} \mid \gamma \in Y(\overline{\K}) \textup{ is a formal Puiseux curve on Y }\}.$ 
\end{itemize}
In fact, a valuation $v \in \V_{G/H}$ belongs to any of the sets in (a), (b) or (c) if and only if the closure of $Y$ in $X_v$, the spherical embedding associated to the ray generated by $v$, intersects the divisor at infinity $D_v$.
\end{Th}
\begin{proof}
The theorem follows from \ref{prop-sph-trop-var-G/H-vs-X_B} and \ref{prop-Vogiannou}.
\end{proof}

\subsection{Analytification and spherical tropicalization map} \label{subsec-analytification}
In this section we briefly recall the notion of a Berkovich analytic space or analytification of a variety $X$. It plays a central role in non-Archimdean geometry. 
As before let $\K = \k((t))$ denote the field of formal Laurent series in an indeterminate $t$. It is equipped with a natural valuation $\ord_t: \K \to \z \cup \{\infty\}$. 

Let $A$ be a finitely generated $\k$-algebra and $X = \Spec(A)$ the corresponding affine variety. 
Let $\tilde{A} = A \otimes_\k \K$.
\begin{Def}[Multiplicative seminorm] \label{def-seminorm}
A function $p: \tilde{A} \to \r_{\geq 0}$ is called a {\it multiplicative seminorm on $A$} if it satisfies the following:
\begin{itemize}
\item[(a)] $p(fg) = p(f)p(g)$,
\item[(b)] $p(\lambda) = \exp(-\ord_t(\lambda))$, 
\item[(c)] $p(f+g) \leq \max(p(f), p(g))$,
for all $f, g \in \tilde{A}$ and $\lambda \in \K$. 
The analytification $X^{\an}$ of $X$ is the collection of all multiplicative seminorms on $\tilde{A}$. We endow $X^{\an}$ with the coarsest topology in which the maps
$X^{\an} \to \r$, $p \mapsto p(f)$, are continuous for every $f \in \tilde{A}$.
\end{itemize}
\end{Def} 

\begin{Rem} \label{rem-seminorm-val}
(1) One shows that the axiom (c) in Definition \ref{def-seminorm}  is equivalent to the triangle inequality. 
(2) For a multiplicative seminorm $p$ let us define 
\begin{equation} \label{equ-val-seminorm}
v_p(f) = -\log(p(f)), 
\end{equation}
for all $f \in A$. One verifies that $v_p: A \to \r \cup \{\infty\}$ is a valuation (in this context it is more convenient to consider a valuation as a map from $A$ to $\r \cup \{\infty\}$ and define the value of $0$ to be $\infty$).
\end{Rem}

There is a natural embedding from $j: X(\K) \hookrightarrow X^\an$ given by restricting to points in $X(\K)$. More precisely, for each point $\gamma \in X(\K)$ we let $p = j(\gamma)$ to be the multiplicative seminorm 
defined by: 
\begin{equation} \label{equ-X-X^an}
j(\gamma)(f) = \exp(-\ord_t(f(\gamma))).
\end{equation}

Now let $X$ be an affine spherical $G$-variety. In the context of multiplicative seminorms it is natural to extend the valuation cone $\V_X$ and define $\hat{\V}_X$ to be the set of all invariant valuations $v: \k(X) \to \r \cup \{\infty\}$.
Recall that to any valuation $v$ on $X$ we can associate a $G$-invariant valuation $\bar{v}$ on $X$ (see Theorem \ref{th-Sumihiro}).
For any $f \in \k(X)$, the valuation $\bar{v}$ is defined by:
$$\bar{v}(f) = v(g \cdot f),$$
for $g \in G$ in general position, i.e. $g$ lies in some Zariski open subset $U_f$ of $G$.
Moreover, $\bar{v}(f) = \min\{ v(g \cdot f) \mid g \in G\}$.

More generally, let $Y \subset X$ be a subvariety that intersects the open $G$-orbit. Let $\pi: A \to \k[Y]$ be the algebra homomorphism induced by the inclusion of $Y$ in $X$. 

For a valuation $v: \k[Y] \to \r \cup \{\infty \}$ we denote by $\bar{v}: A \to \r \cup \{\infty \}$ the valuation defined as follows. For any $f \in A$ let:
\begin{equation} \label{equ-bar-v}
\bar{v}(f) = v( \pi(g \cdot f)),
\end{equation}
for $g$ in some Zariski open subset $U_f$.
Now let $p \in Y^\an$ with the associated valuation $v_p$. Let $\bar{v}_p$ be the $G$-invariant valuation on $\k(X)$ associated to $v_p$ as above.
\begin{Def}[Spherical tropicalization map]
We define the {\it spherical tropicalization map} $\TROP: Y^\an \to \hat{\V}_X$ by: 
$$p \mapsto \bar{v}_p.$$
\end{Def}

\begin{Prop} \label{prop-trop-analytification}
We have the following:
\begin{itemize}
\item[(1)] The map $\TROP: Y^\an \to \hat{\V}_X$ is continuous.
\item[(2)] The map $\TROP$ extends the tropicalization map $\Trop: Y(\K) \to \V_X$ introduced in Section \ref{subsec-tropicalization}. That is, the diagram below commutes:
$$ \xymatrix{
Y(\K)~ \ar@{^{(}->}[rr]^{j} \ar[rd]_\Trop & & Y^\an \ar[ld]^{\TROP} \\
& \hat{\V}_X &\\
}$$
\end{itemize}
\end{Prop}
\begin{proof}
{(1) To prove continuity of $\TROP$ it suffices to show that for any $f \in A$ the map $p \mapsto \bar{v}_p(f)$ is continuous. Take $f \in A$.
Let $M_f$ be the finite dimensional $G$-submodule of $A$ generated by $f$. Let $\{g_1 \cdot f, \ldots, g_s \cdot f\}$ be a finite spanning set for $M_f$ where $g_i \in G$.
Then from the definition of the valuation $\bar{v}_p$ we know that (see Theorem \ref{th-Sumihiro}):
$$\bar{v} _p(f) = \min\{v_p(g_1 \cdot f), \ldots, v_p(g_s \cdot f)\}.$$
By the definition of the topology on $X^\an$ each of the functions $p \mapsto v_p(g_i \cdot f) = -\log(p(g_i \cdot f))$ is continuous. The continuity of $\TROP$ now follows from the fact that 
the minimum of a finite number of continuous functions is continuous. Part (2) of the proposition is a straightforward corollary of \eqref{equ-X-X^an} and the definitions of the maps $\Trop$ and $\TROP$.}
\end{proof}

\section{Some examples}
\subsection{Torus}  \label{subsec-torus}
Let $G = T = (\k^*)^n$ and $H = \{e\}$. The torus $T$ is clearly a $T$-spherical homogeneous space. In this case, for any subvariety 
$Y \subset T$, the tropical variety $\trop(Y)$ coincides with the classical tropical variety of $Y$ and Theorem \ref{th-fundamental} recovers the fundamental theorem of tropical geometry
(for the constant coefficient case). 

\subsection{Punctured affine plane} \label{subsec-punctured-plane}
As in Example \ref{ex-sph-var}(3) consider the spherical variety $X = \A^2 \setminus \{(0,0)\}$ for the natural action of $G=\SL(2,\k)$.  We recall that this action is transitive. The stabilizer of the point $(1, 0)$ is the subgroup $U$ of upper triangular matrices with $1$'s on the diagonal and we can identify $X$ with $G/U$. The coordinate ring $\k[X]$ is just the polynomial algebra $\k[x, y]$.

The coordinate function $y$ is a $B$-eigenfunction and any $B$-eigenfunction is of the form $y^k$, $k \in \z$. Thus $\Lambda_X \cong \z$ and hence the valuation cone $\V_X$ can be identified with $\q$. It is generated (as a cone) by two distinguished $G$-invariant valuations $v_1$ and $v_2$ where $v_1 = -v_2$ regarded as elements of $\q$. As valuations they are given as follows. Let $h \in \k[x, y]$ and write $h = \sum_{i=m}^d h_i$ as sum of its homogeneous components with $h_m, h_d \neq 0$. Then $v_1(h) = m$ and $v_2(h) = -d$, i.e. $v_1$ is the degree of smallest term and $v_2$ is minus the degree.


Let $Y \subset \A^2 \setminus \{(0,0)\}$ be a curve given by an equation $f(x, y) = 0$ where $f$ is a nonconstant polynomial. Let $f = \sum_{i=m}^d f_i$ where $f_i$ is the homogeneous component of $f$ of degree $i$ and $f_m, f_d \neq 0$. A description of the tropicalization of $\Trop(Y)$ in this example is obtained in 
\cite[Example 3.10]{Vogiannou}. One has:
\begin{equation} \label{equ-Trop-SL(2)-example}
\trop(Y) = 
\begin{cases}
\q & m > 0 \\
\q_{\leq 0} & m = 0.\\ 
\end{cases}
\end{equation}
That is, $\Trop(Y) \subset \q$ is the negative ray $\q_{\leq 0}$ if $Y$ does not pass through the origin, and is the whole line $\q$ if it does. We verify the fundamental theorem (Theorem \ref{th-fundamental}) in this example by computing the tropical variety from initial ideals 
and Borel charts (Definition \ref{def-sph-trop-var-X_B} and Proposition \ref{prop-sph-trop-var-G/H-vs-X_B}).

Let $B$ and $B^-$ denote the Borel subgroups of upper triangular and lower triangular matrices respectively.
It is easy to see that the $B$-orbit and $B^-$-orbit of the point $(0,1)$ are $X_B = \A^2 \setminus \{y \neq 0\}$ and $X_{B^-} = \A^2 \setminus \{x \neq 0\}$.
Thus the coordinate rings $\k[X_B]$ and $\k[X_{B^-}]$ are $\k[x, y, y^{-1}]$ and $\k[x, y, x^{-1}]$ respectively.
Clearly the action of $G$ on $X$ extends to the whole projective plane $\p^2$. 
One can verify that the condition in the proof of Proposition \ref{prop-sph-trop-var-finite-B} is satisfied for 
the complete spherical embedding $\p^2$ and the collection of Borel subgroups
$\{B, B^-\}$. That is, every $G$-orbit $O \subset \p^2$ is covered by the open $B$-orbit and the open $B^-$-orbit contained in $O$.

First consider the case $v = v_1$. One can check that
$v \in \trop_{B}(I)$ if and only if $f_m$ is neither a constant nor a power of $y$. Similarly, $v \in \trop_{B^-}(I)$ if and only if 
$f_m$ is neither a constant nor a power of $x$. Putting these together we see that $v \in \trop(Y)$ if and only if $f_m$ is not a constant. 
The case $v = v_2$ can be dealt with in a similar fashion. In this case we have $v \in \trop_B(I)$ if and only if $f_d$ is not a power of $y$, and 
$v \in \trop_{B^-}(I)$ if and only if $f_d$ is not a power of $x$. It thus follows that $\q_{\leq 0}$ is always contained in $\trop(Y) = \trop_B(I) \cup \trop_{B^-}(I)$. This recovers \eqref{equ-Trop-SL(2)-example}.

\subsection{Group $G$ with $(G \times G)$-action} \label{subsec-GxG}
{Following Example \ref{ex-sph-var}(5) and Example \ref{ex-non-Arch-Cartan-decomp}
consider the left-right action of $G \times G$ on $G$. Recall that $X = G$ is a spherical $(G \times G)$-homogeneous space. In fact, we can identify $G$ with 
the homogeneous space $(G \times G) / G_{\textup{diag}}$ where $G_{\textup{diag}} = \{ (g, g) \mid g \in G\}$. 
Let us consider the case where $X = G = \GL(2, \k)$. We identify the valuation cone of $X$ with: $$\V_X = \{(x, y) \mid x \geq y\}.$$ 
We denote a general element of $G$ by a matrix $g = \left[ \begin{matrix} a & b \\ c & d\end{matrix} \right]$. The coordinate ring $\k[G]$ is 
then the localization of the polynomial algebra $\k[a, b, c, d]$ at $ad-bc$, i.e. $\k[G] = \k[a, b, c, d, (ad-bc)^{-1}]$.
We compute the spherical tropical variety for two hyperplanes $Z_1$ and $Z_2$ in the open $(B \times B)$-orbit in $G$ where $B$ is the subgroup of 
upper triangular matrices. We denote this open Borel orbit by 
$X_{B \times B}$. It is easy to see that $X_{B \times B} = \{ g \in \GL(2, \k) \mid c \neq 0\}$. Thus 
$\k[X_{B \times B}] = \k[a, b, c, d, c^{-1}, (ad-bc)^{-1}]$. 

Let $Z_1$ and $Z_2$ to be the hyperplanes in $X_{B \times B}$ defined by $c=1$ and $d=1$ respectively. By definition (Definition \ref{def-sph-trop-var-X_B}) we have $\trop_{B \times B}(Z_1)$ (respectively $\trop_{B \times B}(Z_2)$) is the set of all $v \in \V_X$ such that the initial form of $c-1$ (respectively $d-1$) is not invertible. We note that $c$ is invertible in $\k[X_{B \times B}]$ while $d$ is not. 
One shows that $\trop_{B \times B}(Z_1)$ is the ray $R_1 = \{(x, 0) \mid x \geq 0\}$ and 
$\trop_{B \times B}(Z_2)$ is the angle between the two rays $R_1 = \{(x, 0) \mid x \geq 0\}$ and $R_2 = \{(x, x) \mid x \leq 0\}$. 

Let $M(2, \k)$ denote the vector space of $2 \times 2$ matrices.
Let $G \times G$ act on $M(2, \c) \times \k$ where it acts by multiplication from left and right on the first component $M(2, \k)$ and trivially on 
the second component $\k$. Projectivizing this action we get a $(G \times G)$-action on the projective space $\p^4$. Finally, this action induces a $(G \times G)$-action on the blowup at the origin $\textup{Bl}_0(\p^4)$. Let $\overline{X}$ denote this blowup. One verifies that it is a 
complete toroidal spherical embedding of $G \times G$. Moreover, it contains $3$ codimension $1$ $(G \times G)$-orbits $O_1$, $O_2$ and $O_3$. In the colored fan of this spherical embedding, $O_1$ and $O_2$ correspond to the rays $R_1$ and $R_2$ in the valuation cone respectively. One checks that the closure of $Z_1$ in $\overline{X}$ intersects the open Borel orbit in $O_1$ but it does not intersect the open Borel orbit in $O_2$. On the other hand, the closure of $Z_2$ intersects both of these Borel orbits. This is in agreement with Theorem \ref{th-sph-trop-var-X_v_B}.

\section{Spherical amoebas and Cartan decomposition} \label{sec-sph-amoeba}
There is a logarithm map on the complex algebraic torus $\Log_t: (\c^*)^n \to \r^n$ defined by:
\begin{equation} \label{equ-log-map-torus}
\Log_t(z_1, \ldots, z_n) = (\log_t(|z_1|), \ldots, \log_t(|z_n|)).
\end{equation}
Clearly the inverse image of every point is an $(S^1)^n$-orbit in $(\c^*)^n$. Note that $(S^1)^n$ is a 
maximal compact subgroup in $(\c^*)^n$.
For a subvariety $Y \subset (\c^*)^n$, the amoeba $\mathcal{A}_t(Y)$ of $Y$ is the image of $Y$ in $\r^n$ under the logarithm map $\Log_t$ (\cite{GKZ}). 
It is well-known that as $t \to 0$ the amoeba $\mathcal{A}_t(Y)$ approaches the tropical variety of $Y$. 
The goal of this section is to extend the above to spherical homogeneous spaces.

Let $G$ be a complex connected reductive algebraic group and $G/H$ a spherical homogeneous space. 
Let $T_{G/H}$ be the torus associated to $G/H$, i.e., $T_{G/H}$ is the torus whose lattice of characters is $\Lambda_{G/H}$.
One can see that $T_{G/H}$ can be identified with $T / T \cap H$ for a maximal torus $T \subset G$. Thus $T_{G/H}$ also can be identified with the $T$-orbit of $eH \in G/H$.

We now take the ground field to be $\k = \c$.
We consider the exponential map $\exp: \Lie(T_{G/H}) \to T_{G/H} \subset G/H$. As usual, the valuation cone lies in the vector space $\mathcal{Q}_{G/H} = \Hom(\Lambda_{G/H}, \q)$ which in turn we consider as a subset of $\Lie(T_{G/H})$. The image $\exp(\V_{G/H})$ of the valuation cone thus naturally sits in $T_{G/H} \subset G/H$. Let $T_{real, G/H} \subset T_{G/H}$ be the closed subgroup of $T_{G/H}$ corresponding to the the real subalgebra 
$\mathcal{Q}_{G/H} \otimes \r \subset \Lie(T_{G/H})$. We let $\Log_t$ denote the inverse of the map $T_{real, G/H} \to \mathcal{Q}_{G/H}$ given by 
$\xi \mapsto \exp(t\xi)$. 
For the rest of this section we make the following assumption:
\begin{Ass}[Archimedean Cartan decomposition for a spherical homogeneous space] \label{ass-Arch-Cartan-sph}
There exists a maximal compact subgroup $K$ of $G$ which is a real algebraic subgroup such that each $K$-orbit in $G/H$ intersects the image of the valuation 
cone $\exp(\V_{G/H})$ at a unique point. 
\end{Ass}

We can then define the map $\L_t: G/H \to \V_{G/H}$ by:
$$x \mapsto \Log_t((K \cdot x) \cap \exp(\V_{G/H}))  \in \V_{G/H},$$
that is, first we intersect the orbit $K \cdot x$ with $\exp(V_{G/H})$ and then map it to the valuation cone by the logarithm map $\Log_t$. We call $\L_t$ a {\it spherical logarithm map}.
\begin{Def}[Spherical amoeba] \label{def-spherical-amoeba}
Let $Y \subset G/H$ be a subvariety. We denote the image of $Y$ under the map $\L_t$ by $\mathcal{A}_t(Y)$ and call it the {\it spherical amoeba} of the 
subvariety $Y$.
\end{Def}

\begin{Rem} \label{rem-Batyrev-conj}
Motivated by the Cartan decomposition and the Iwasawa decomposition from Lie theory, the authors conjectured that 
there exist Archimedean Cartan decompositions for spherical varieties. That is, for a complex spherical homogeneous space $G/H$ one can 
find a maximal compact subgroup $K$ such that each $K$-orbit in $G/H$ intersects the image of the valuation cone at a unique point.
Later we learned that in fact Victor Batyrev has conjectured the same statement some years ago. As far as we know no proof or counterexample is known.
A result in this direction can be found in \cite{Knop-Cartan}.  
\end{Rem}

The purpose of rest of  this section is to prove that: {\it the spherical amoeba approaches the spherical tropical variety}.
Let $(\gamma, \theta_0)$ be a germ of an algebraic curve $\gamma: \Gamma \dashrightarrow G/H$. 
Take a local uniformizing parameter $t$ for $\O_{\Gamma, \theta_0}$ and let us consider $\gamma$ as a Laurent series in the 
variable $t$. Then there exists $r > 0$ such that 
$\gamma(t)$ is convergent (in the classical topology) for all $t \in \c$ with $0 < |t| < r$ (this is because there exists a neighborhood $U$ of the point 
$\theta_0$, in the classical topology on $\Gamma$, such that $\gamma$ is defined at every point in $U \setminus \{\theta_0\}$). 
For $t$ sufficiently small let us use the Archimedean Cartan decomposition (see Assumption \ref{ass-Arch-Cartan-sph} above) to write 
\begin{equation} \label{equ-gamma-k-tau}
\gamma(t) = k(t) \tau(t),
\end{equation}
where $k(t) \in K$ and $\tau(t) \in \exp(\V_{G/H})$.
Let $v = \bar{v}_\gamma \in \V_{G/H}$ denote the $G$-invariant valuation associated to the 
curve $\gamma$. 

\begin{Th} \label{th-amoeba-trop-var-curve}
With notation as above we have: $$\lim_{t \to 0} \Log_t(\tau(t)) = v.$$
\end{Th}
\begin{proof}
Let us embed $G/H$ in an affine space $\c^N$ and $G$ in $\textup{GL}(N, \c)$ such that the action of $G$ on $G/H$ is the restriction of the natural 
action of $\textup{GL}(N, \c)$ on $\c^N$.
We claim that in \eqref{equ-gamma-k-tau}, we can choose $k(t)$ such that it is real algebraic regarded as a map $k: \c = \r^2 \to K$.
This follows from two facts: Firstly the action map $K \times T_{G/H}  \to G/H$ is surjective by Assumption \ref{ass-Arch-Cartan-sph}. 
Secondly the base change of a surjective algebraic map is also surjective.
That is, the action map $(K \times T_{G/H})(\overline{\c(t, \bar{t})}) \to G/H(\overline{\c(t, \bar{t})}) \supset G/H(\overline{\c(t)})$ is surjective. 
Here $\overline{\c(t)}$ is the algebraic closure of the field of rational functions and 
$\overline{\c(t, \bar{t})}$ is the algebraic closure of the field of rational functions in $t$ and $\bar{t}$.


Let us restrict $t$ to the real interval $(0, r)$. 
Let $\tilde{k}(t)$ be the holomorphic curve defined in some punctured neighborhood $\{ t \in \c \mid 0 < |t| < \epsilon\}$  
such that $\tilde{k}(t) = k(t)$ for real $t$ with $0 < t < \epsilon$ (to get $\tilde{k}(t)$ replace $\bar{t}$ with $t$ in the series expansion of $k(t)$ in $t$ and $\bar{t}$ near $0$). Now since $K$ is compact we know that $k(t)$ is bounded for all values of $t$. But $k(t)$ is an algebraic function in $t$ and $\bar{t}$. 
This implies that $k(t)$ has a limit as $t$ goes to $0$.

Since $\tilde{k}(t)$ is represented by the same collection of series as $k(t)$, where we replace $\bar{t}$ with $t$ everywhere, we conclude that $\tilde{k}(t)$ is holomorphic at $0$. 
From this it follows that $\tilde{\tau}(t) := \tilde{k}(t)^{-1} \cdot \gamma(t)$ is meromorphic 
at $t = 0$. Also $\tilde{\tau}(t) = \tau(t)$ for $0 < t < \epsilon$.  

Note that by assumption $\tau(t)$ lies in the image $\exp(\V_{G/H})$ of the valuation cone, 
hence $\tilde{\tau}(t)$ lies in the torus $T_{G/H}(\c)$ regarded as a subset of $G/H(\c)$ (the complex curve $\tilde{\tau}$ intersects the complex manifold $T_{G/H}(\c)$ in a non-isolated set and hence should be contained in $T_{G/H}(\c)$). Now from the fact that $\tilde{k}(t)$ is holomorphic at $t=0$ we see that $\tilde{\tau}(t)$ is 
in the $G(\mathcal{O})$-orbit of $\gamma(t)$. We can write $\tilde{\tau}(t) = t^d \theta(t)$ where $t^d$ is a one-parameter subgroup in 
$\exp(\V_{G/H}) \subset T_{G/H}$ and $\theta(t) \in T_{G/H}(\mathcal{O})$. We note that by the uniqueness part of the 
non-Archimedean Cartan decomposition (Theorem \ref{th-non-Arch-Cartan-decomp}) $d$ should coincide with $v = \bar{v}_\gamma$, the $G$-invariant valuation corresponding to $\gamma$. 
Then $$\log_t(\tilde{\tau}(t)) = v + \log_t(\theta(t)).$$
Since $\theta(t) \in T_{G/H}(\mathcal{O})$, it is bounded in a neighborhood of $0$ and hence $\lim_{t \to 0} \log_t(\theta(t)) = 0$. This proves that 
$\lim_{t \to 0} \log_t(\tilde{\tau}(t)) = v$ and hence $\lim_{t \to 0} \log_t(\tau(t)) = v$ as required.
\end{proof}


Next we show that every point in the tropical variety is a limit of points from the amoeba. 
\begin{Th} \label{th-trop-var-limit-amoeba}
Let $Y \subset G/H$ be a subvariety. Let $v \in \Trop(Y)$ be a point in the spherical tropical variety of $Y$. 
Then there exists a formal curve $\gamma \in Y(\K)$ with nonzero radius of convergence and such that: $$\lim_{t \to 0} \L_{t}(\gamma(t)) = v.$$ 
\end{Th}
\begin{proof}
By Theorem \ref{th-approx-algebraic} we can find a point $\gamma$ of $Y$ over the field of algebraic functions $\overline{\c(t)}$ with 
$\Trop(\gamma) = v$. The theorem now follows from Theorem \ref{th-amoeba-trop-var-curve}.
\end{proof}

\begin{Ex}[Torus]  \label{ex-amoeba-torus}
As in Example \ref{ex-sph-var}(1) let $G = T = (\c^*)^n$. In this case, for any subvariety $Y \subset T$, the 
tropical variety $\Trop(Y)$ coincides with the usual tropical variety of $Y$. 
If we take the maximal compact subgroup $K = (S^1)^n$, i.e. the compact torus in $T$, then the corresponding logarithm map is the 
usual logarithm map from $T$ to $\r^n$. Hence $\mathcal{A}_t(Y)$ is the usual amoeba of the subvariety $Y$.
\end{Ex}

\begin{Ex}[Punctured affine plane] \label{ex-amoeba-punctured-plane}
As in Example \ref{ex-sph-var}(3) consider the spherical variety $X = \c^2 \setminus \{(0,0)\}$ for the natural action of $G = \SL(2, \c)$.  
We recall that this action is transitive and $X$ can be identified with the
homogeneous space $G/U$ where $U$ is the subgroup of upper triangular matrices with $1$'s on the diagonal.
We explicitly describe the spherical logarithm map in this example. Let the maximal compact subgroup $K$ be $\textup{SU}(2) \subset \SL(2, \c)$. We have the 
Iwasawa decomposition $G = KTH$. 
Consider a point $p = (x, y) \in X = \c^2 \setminus \{(0, 0)\}$. Then:
$$ \left[ \begin{matrix} x & \frac{-y}{|x|^2 + |y|^2} \\ y & \frac{x}{|x|^2 + |y|^2} \end{matrix}\right] \left[\begin{matrix} 1 \\ 0 \end{matrix}\right] 
= \left[\begin{matrix} x \\ y \end{matrix} \right].$$
The Iwasawa decomposition of the matrix above into a product of a unitary matrix and an upper triangular matrix is:
$$ \left[ \begin{matrix} x & \frac{-y}{|x|^2 + |y|^2} \\ y & \frac{x}{|x|^2 + |y|^2} \end{matrix}\right] = \frac{1}{\sqrt{|x|^2+|y|^2}} 
\left[ \begin{matrix} x & -y \\ y & x \end{matrix}\right] \left[ \begin{matrix}  \sqrt{|x|^2+|y|^2} & 0 \\ 0 & \frac{1}{\sqrt{|x|^2+|y|^2}} \end{matrix}\right].$$
This shows that the unique point of intersection of the $K$-orbit of $p$ with $\exp(\V_{G/H}) \subset T_{G/H}$ is represented by the diagonal matrix:
$$\left[ \begin{matrix}  \sqrt{|x|^2+|y|^2} & 0 \\ 0 & \frac{1}{\sqrt{|x|^2+|y|^2}} \end{matrix}\right].$$ 
Thus the spherical logarithm map, corresponding the the choice of $K = \textup{SU}(2, \c)$ is given by $$\L_t(p) = \log_t(||p||),$$ where $||p|| = \sqrt{|x|^2 + |y|^2}$ is the length of $p$. 

Now consider a curve $\gamma(t) = (\gamma_1(t), \gamma_2(t))$ with nonzero radius of convergence. We observe that as $t$ approaches $0$ the order in $t$ of the 
expression $\sqrt{|\gamma_1(t)|^2 + |\gamma_2(t)|^2}$ is equal to $\min(\ord_t(\gamma_1(t)), \ord_t(\gamma_2(t)))$.
This verifies Theorem \ref{th-amoeba-trop-var-curve} in this example.
\end{Ex}

\begin{Ex}[Homogeneous space $G/U$] \label{ex-amoeba-G/U}
As in Example \ref{ex-sph-var}(4) consider the spherical homogeneous space $G/U$ where $U$ is a maximal unipotent subgroup of $G$.
Again take the base field to be $\c$ and take a maximal compact subgroup $K$ such that $K \cap T$ is a maximal compact subgroup of the torus $T$. 
Then by the Iwasawa decomposition we have $G = KTU$. If we write $g \in G$ as $kau$, where $k \in K$, $a \in T$ and $u \in U$ as in the Iwasawa decomposition then the spherical logarithm 
map $\L_t$ sends the point $gU \in G/U$ to $\Log_t(a)$ where $\Log_t$ denotes the usual logarithm map for the torus $T$ (see \eqref{equ-log-map-torus}).  
\end{Ex}

\begin{Ex}[Group $G$ with $G \times G$-action] \label{ex-amoeba-GxG}
As in Example \ref{ex-sph-var}(5) consider $X=G$ as a spherical variety for the left-right action of $G \times G$.
Take the base field to be $\k = \c$ and let $G = \GL(n, \c)$ or $\SL(n, \c)$. Also let $B$ and $T$ denote the subgroups of upper triangular and diagonal matrices respectively.
Moreover, let $K = \textup{U}(n)$ or $\textup{SU}(n)$ be the maximal compact subgroup of unitary matrices. 

In this case, the (Archimedean) Cartan decomposition (Assumption \ref{ass-Arch-Cartan-sph}) is the well-known singular value decomposition theorem. 
Recall that if $A$ is an $n \times n$ complex matrix, the singular value decomposition states that $A$ can be written as: $$A = U_1DU_2,$$ where $U_1$, $U_2$
are $n \times n$ unitary matrices and $D$ is diagonal with nonnegative real entries. In fact, the diagonal entries of $d$ are the eigenvalues of 
the positive semi-definite matrix $\sqrt{AA^*}$ where $A^* = \bar{A}^t$. The diagonal entries of $A$ are usually referred to as the {\it singular values of $A$}.
On the other hand, in this case, the non-Archimdean Cartan decomposition (Theorem \ref{th-non-Arch-Cartan-decomp}) 
is the Smith normal form theorem (see also Example \ref{ex-non-Arch-Cartan-decomp}). 
Let $A(t)$ be an $n \times n$ matrix whose entires $A_{ij}(t)$ are Lauren series in $t$ and over $\c$. 
We recall that the Smith normal form theorem (over the ring of formal power series which is a PID) states that $A(t)$ can be written as:
$$A_1(t) \tau(t) A_2(t),$$
where $A_1(t)$, $A_2(t)$ are $n \times n$ matrices with power series entries and $\tau(t)$ is a diagonal matrix of the form 
$\tau(t) = \textup{diag}(t^{v_1}, \ldots, t^{v_n})$ for integers $v_1, \ldots, v_n$. The integers $v_1, \ldots, v_n$ are usually called the 
{\it invariant factors of $A(t)$}. 


As an application of Theorem \ref{th-amoeba-trop-var-curve} we thus obtain the following.
\begin{Cor}[Invariant factors versus singular values of a matrix] \label{cor-inv-fact-sing-values}
Let $A(t)$ be an $n \times n$ matrix whose entries $A_{ij}$ are algebraic functions in $t$. For sufficiently small $t \neq 0$, let $d_1(t) \leq \cdots \leq d_n(t)$ denote the singular values of 
$A(t)$ ordered increasingly. Also let $v_1 \geq \cdots \geq v_n$ be the invariant factors of $A(t)$ ordered decreasingly. We then have:
$$\lim_{t \to 0} (\log_t(d_1(t)), \ldots, \log_t(d_n(t))) = (v_1, \ldots, v_n).$$
\end{Cor}

{The above can also be proved using the Hilbert-Courant minimax principle.}

Let us look at a concrete example in this case. As in Section \ref{subsec-GxG} 
consider the curve $Y$ in $\GL(2, \c)$ defined by the ideal:
$$I = \langle x_{11} - x_{12} - 1, x_{12} - x_{21}, x_{22} \rangle.$$ 
The spherical tropical variety of $Y$ is computed in 
\cite[Example 5.3]{Vogiannou}.
Using the parametrization
$$Y = \{ \left[ \begin{matrix} t+1 & t \\ t & 0 \end{matrix} \right] \mid t \in \c \},$$
one can compute the spherical amoeba of $Y$. It is plotted in Figure \ref{fig-GL(2)-amoeba}. 
\begin{figure}
\includegraphics[width=10cm]{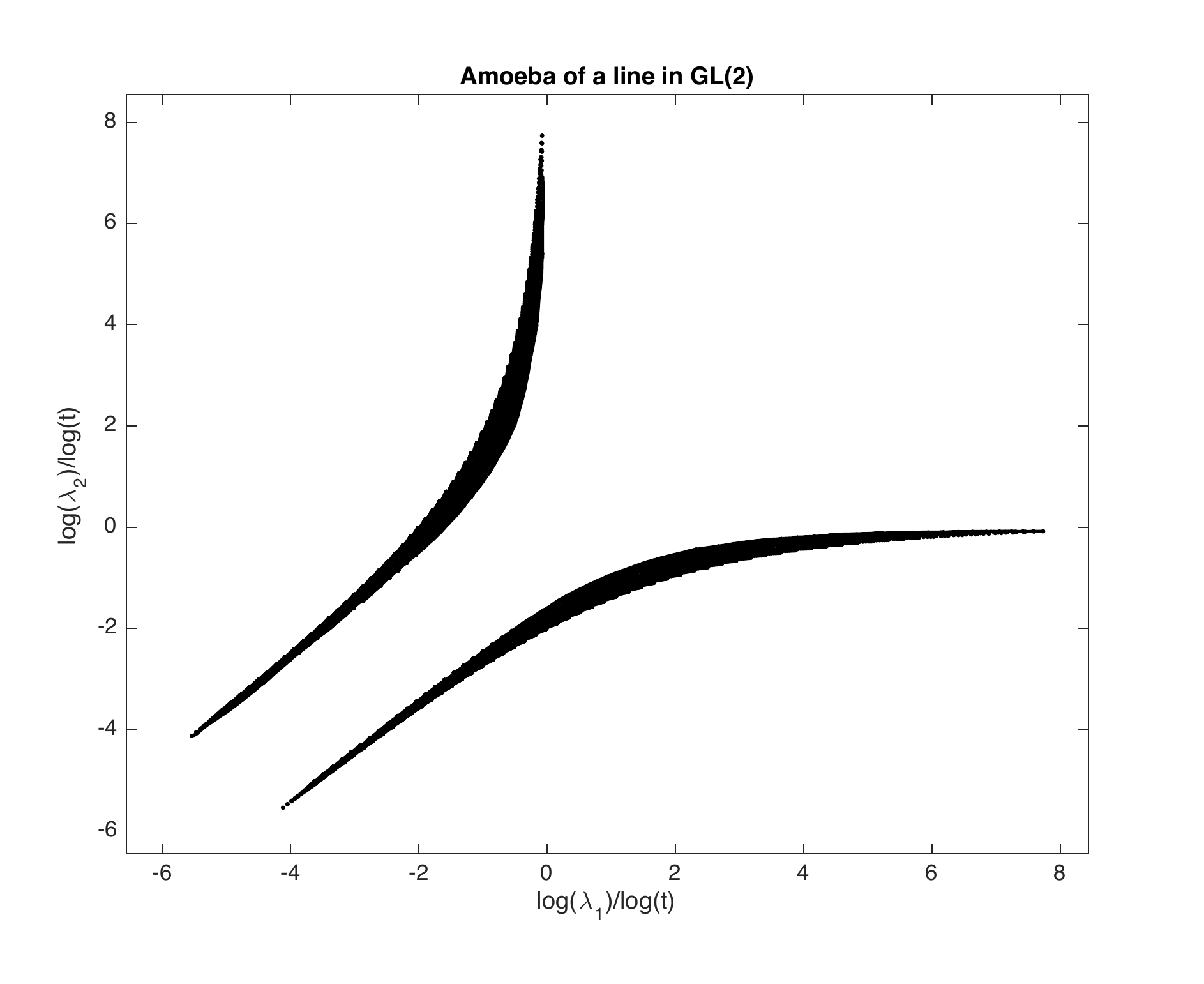} 
\caption{An approximate picture of the spherical amoeba of a line in $\GL(2, \c)$ (in fact, the picture shows the union of the images of the amoeba under the Weyl group of $\GL(2, \c)$).} 
\label{fig-GL(2)-amoeba} 
\end{figure} 
\end{Ex}

\end{document}